\documentclass[a4, 12pt, 
]{amsart}

\usepackage[all]{xy}
\usepackage{mathrsfs, amsthm, amscd, amsmath, amssymb, graphicx}

\usepackage[dvipdfmx]{hyperref}

\newtheorem{theo}{Theorem}[section]

\newtheorem{lemm}[theo]{Lemma}
\newtheorem{cor}[theo]{Corollary}
\newtheorem{claim}[theo]{Claim}
\newtheorem{prob}[theo]{Problem}
\newtheorem{conj}[theo]{Conjecture}

\numberwithin{equation}{section}

\theoremstyle{definition}
\newtheorem{defi}[theo]{Definition}

\newtheorem{step}{Step}

\theoremstyle{remark}
\newtheorem{rem}[theo]{Remark}

\newcommand{\Ker}[0]{\operatorname{Ker}}
\newcommand{\Spn}[0]{\operatorname{Spn}}
\newcommand{\Hom}[0]{\operatorname{Hom}}
\newcommand{\Alb}[0]{\operatorname{Alb}}

\newcommand{\End}[0]{\operatorname{End}}

\newcommand{\deldel}{\sqrt{-1}\partial \overline{\partial}}

\newcommand{\Rur}[3]{R_{#1}(#2, \bar{#2}, #3, \bar{#3})}

\usepackage{geometry}
\begin{document}

\title[K\"ahler manifolds with semi-positive holomorphic sectional curvature]
{On morphisms of compact K\"ahler manifolds \\
with semi-positive holomorphic sectional curvature}

\author{Shin-ichi MATSUMURA}

\address{Mathematical Institute, Tohoku University, 
6-3, Aramaki Aza-Aoba, Aoba-ku, Sendai 980-8578, Japan.}

\email{{\tt mshinichi-math@tohoku.ac.jp, mshinichi0@gmail.com}}

\date{\today, version 0.01}

\renewcommand{\subjclassname}{%
\textup{2010} Mathematics Subject Classification}
\subjclass[2010]{Primary 53C25, Secondary 32Q10, 14M22.}

\keywords
{Holomorphic sectional curvature, 
Bisectional curvature, 
Scalar curvature, 
Partially positive curvature, 
Semi-positivity of curvature,
RC positivity, 
Truly flat tangent vectors, 
Structure theorems, 
Uniformization theorems, 
Maximal rationally connected fibrations, 
Rationally connectedness, 
Albanese maps, 
Abelian varieties, 
Complex tori. }

\maketitle

\begin{abstract}
In this paper, with the aim of establishing a structure theorem 
for a compact K\"ahler manifold $X$ with semi-positive holomorphic sectional curvature, 
we study a morphism $\phi: X \to Y$ to 
a compact K\"ahler manifold $Y$ with pseudo-effective canonical bundle. 
We prove that the morphism $\phi$ is always smooth (that is, a submersion), 
the image $Y$ admits a finite \'etale cover $T \to Y$ by a complex torus $T$, 
and further that all the fibers are isomorphic when $X$ is projective. 
Moreover, by applying a modified method to maximal rationally connected fibrations, 
we show that $X$ is rationally connected,   
if $X$ is projective and $X$ has no truly flat tangent vectors at some point 
(which is satisfied when the holomorphic sectional curvature is quasi-positive). 
This result gives a generalization of Yau's conjecture. 
As a further application, we obtain a uniformization theorem 
for compact K\"ahler surfaces with semi-positive holomorphic sectional curvature. 
\end{abstract}


\section{Introduction}\label{Sec-1}
One of the fundamental and important problems in differential geometry 
is to establish structure theorems or classifications 
for varieties satisfying certain curvature conditions. 
The Frankel conjecture, 
which has been solved by Siu-Yau in \cite{SY80}  and by Mori in  \cite{Mor79}, 
states that a smooth  projective variety with {\textit{positive bisectional curvature}} 
is isomorphic to the projective space (see \cite{Mor79} for the Hartshorne conjecture). 
As one of the extensions of the Frankel conjecture, 
it is a significant problem to consider the geometry 
of {\textit{semi-positive bisectional curvature}} 
(or more generally the geometry of nef tangent bundles). 
In their paper \cite{HSW81}, 
Howard-Smyth-Wu studied a structure theorem for a compact K\"ahler manifold $M$ 
with semi-positive bisectional curvature  
and they showed that $M$ can be decomposed into  
a \lq \lq flat" manifold $B$ and a \lq \lq quasi-positively curved" manifold $M'$ 
(see also \cite{CG71} and \cite{CG72}). 
Precisely speaking, they proved that $M$ admits a locally trivial morphism $f : M \to B$ 
to a flat K\"ahler manifold $B$ such that 
the fiber $M'$ of $f$ is a smooth projective variety with quasi-positive bisectional curvature. 
The flat manifold $B$ is a complex torus up to finite \'etale covers, 
and thus the geometry of $M$ can be reduced to 
the smooth projective variety $M'$ with quasi-positive bisectional curvature, 
thanks to their structure theorem. 
Further, it can also be  proven that 
the universal cover of $M$ is isomorphic to 
the product $M' \times \mathbb{C}^m$
(see \cite{DPS94} for compact K\"ahler manifolds with nef tangent bundle). 
After the work of Howard-Smyth-Wu, in his paper \cite{Mok88},  
Mok studied smooth projective varieties with quasi-positive bisectional curvature and 
he showed that the fiber $M'$ of $f$ is isomorphic to the product of 
projective spaces and compact Hermitian symmetric manifolds 
(see \cite{CP91} for the Campana-Peternell conjecture).

This paper is devoted to studies of compact K\"ahler manifolds 
whose  {\textit{holomorphic sectional curvature}} is semi-positive or quasi-positive, 
motivated by generalizing Howard-Smyth-Wu's structure theorem for manifolds with semi-positive bisectional curvature 
and Mok's result for the geometry of quasi-positive bisectional curvature. 
\vspace{0.2cm}

The first contribution of this paper is 
concerned with the solution and its generalization 
of the following conjecture posed by Yau in \cite{Yau82}, 
which gives a relation between the \lq \lq strict" positivity of holomorphic sectional curvature 
and the geometry of $X$ (rationally connectedness). 
Yau's conjecture can be seen as an analogy of Mok's result 
in the studies of holomorphic sectional curvature
and it corresponds to the geometry of the fiber $M'$ 
appearing in Howard-Smyth-Wu's structure theorem.

\begin{conj}[Yau's conjecture for projective varieties]\label{conj-Yau}
If a smooth projective variety $X$ admits a K\"ahler metric with positive holomorphic sectional curvature, 
then $X$ is rationally connected $($that is,  two arbitrary points can be connected by a rational curve$)$. 
\end{conj}

In their paper \cite{HW15}, Heier-Wong considered Yau's conjecture  for projective varieties 
under the weaker assumption that the holomorphic sectional curvature is quasi-positive 
(that is, it is semi-positive everywhere and positive at some point). 
We emphasize that it is essentially important to consider Yau's conjecture for quasi-positive holomorphic sectional curvature
from the viewpoint of structure theorems, 
since the bisectional curvature of the fiber $M'$ is quasi-positive, but not necessarily positive everywhere. 
Yang affirmatively solved Yau's conjecture even for the case of compact K\"ahler manifolds 
by introducing the notion of RC positivity in \cite{Yan17a} 
(see \cite{Yan18c} and references therein for recent progress of RC positivity), 
but it seems to be quite difficult to apply his method to the case of quasi-positive holomorphic sectional curvature.

In this paper, 
we obtain a generalization of Yau's conjecture, Heier-Wong's result, and Yang's result 
(see Theorem \ref{thm-mai}), 
by using an idea in \cite{HW15} and by developing techniques for a partial positivity and certain flatness. 
This theorem can be seen as a version of Mok's result for holomorphic sectional curvature, 
and further it gives a more precise relation 
between the positivity of holomorphic sectional curvature 
and the dimension of images of maximal rationally connected (MRC for short) fibrations 
(which measures how far $X$ is from rationally connectedness). 
See \cite{Cam92} and \cite{KoMM92} for MRC fibrations. 

\begin{theo}\label{thm-mai}
Let $(X, g)$ be a compact K\"ahler manifold such that $X$ is projective and 
the holomorphic sectional curvature  is semi-positive. 
Let  $\phi: X \dashrightarrow Y$ be a MRC fibration of $X$. 
Then we have  
$$
\dim X -\dim Y \geq n_{{\rm{tf}}}{(X, g)}. 
$$
In particular, the manifold $X$ is rationally connected 
if $n_{{\rm{tf}}}{(X, g)}=\dim X$ $($which is satisfied if the holomorphic sectional curvature is quasi-positive$)$. 
\end{theo}

Here the invariant $n_{{\rm{tf}}}{(X, g)}$ is defined by 
$$
n_{{\rm{tf}}}{(X, g)}:=\dim X - \inf_{p \in X} \dim V_{{\rm{flat}},p},  
$$
where $V_{{\rm{flat}},p}$ is the subspace of the tangent space $T_{X,p}$ at $p$
consisting of all the truly flat tangent vectors (see subsection \ref{subsec-2-2} for the precise definition). 
The invariant $n_{{\rm{tf}}}{(X, g)}$ can be seen as an analogue of the numerical Kodaira dimension 
in terms of truly flat tangent vectors introduced in \cite{HLWZ}. 
The condition of $n_{{\rm{tf}}}{(X, g)}=\dim X$ 
(that is, there is no truly flat tangent vectors at some point) 
is a weaker assumption than the quasi-positivity, 
but it works in a more flexible manner from the viewpoint of our argument. 
\vspace{0.2cm}

The second contribution of this paper 
is a partial answer for the following structure conjecture. 
The following conjecture, which is a revised version of \cite[Conjecture 1.1]{Mat18}, 
asks a structure theorem for compact K\"ahler manifolds 
with semi-positive holomorphic sectional curvature and 
it can be seen as a natural generalization of Howard-Smyth-Wu's structure theorem.

In this paper, we affirmatively solve Conjecture \ref{conj-str} 
under the assumption that a MRC fibration of a smooth projective variety $X$ 
can be chosen to be a morphism without indeterminacy locus (see Theorem \ref{thm-main}). 
This assumption is satisfied when  $X$ has the nef anti-canonical bundle 
by the deep result of \cite{CH17} (see Corollary \ref{cor-CH}). 
Moreover we solve Conjecture \ref{conj-str} for compact K\"ahler surfaces without any assumptions 
(see Corollary \ref{cor-main}). 
 

\begin{conj}[{cf. \cite{HSW81} and \cite[Conjecture 1.1]{Mat18}}]\label{conj-str}
Let $X$ be a compact K\"ahler manifold with semi-positive holomorphic sectional curvature. 
\vspace{0.2cm}\\
$(1)$
Then there exists a smooth morphism $X \to Y$ with the following properties\,$:$
\begin{itemize}
\item The morphism $X \to Y$ is locally trivial $($that it, all the fibers $F$ are isomorphic$)$. 
\item The fiber $F$ is projective and rationally connected. 
\item $Y$ is a compact K\"ahler manifold with flat metric. 
\end{itemize}
In particular, there exist a complex torus $T$ and a finite \'etale cover $T \to Y$ such that 
the fiber product $X^*:=X \times_Y T $ admits 
a locally trivial morphism $X^*=X \times_Y T  \to T$ to the complex torus $T$ 
with the rationally connected fiber $F$ and that it satisfies the following commutative diagram\,$:$
\begin{equation*}
\xymatrix@C=40pt@R=30pt{
X^{*}=X \times_Y T \ar[d] \ar[r]^{} & T\ar[d]^{} \\ 
X \ar[r]^{}  &  Y.\\   
}
\vspace{0.1cm}
\end{equation*}
$(2)$ Moreover we have the decomposition 
$$
X_{\rm{univ}} \cong \mathbb{C}^m \times F, 
$$
where $X_{\rm{univ}}$ is the universal cover of $X$ and 
$F$ is the rationally connect fiber. 

In particular, the fundamental group of $X$ is an extension of a finite group by $\mathbb{Z}^{2m}$. 
\end{conj}

When we approach to the above conjecture in the case of $X$ being projective, 
it seems to be the right direction to study 
a MRC fibration 
$\phi: X \dashrightarrow Y$ of $X$, 
based on the strategy explained in \cite{Mat18}. 
We remark that MRC fibrations are almost holomorphic maps  
(that is, dominant rational maps with compact general fibers) and
they are uniquely determined up to birational equivalence. 
It can be seen that we can always choose a MRC fibration $\phi: X \dashrightarrow Y$ 
such that the image $Y$ is smooth by taking a resolution of singularities and 
that the image $Y$ has the pseudo-effective canonical bundle by \cite[Theorem 1.1]{GHS03} and  \cite{BDPP}.

The following theorem, which is one of the main results of this paper, 
reveals a detailed structure of morphisms whose domain has semi-positive holomorphic sectional curvature. 
Theorem \ref{thm-main} is formulated for MRC fibrations of projective varieties and  
Albanese maps of K\"ahler manifolds. 
By applying Theorem \ref{thm-main} to MRC fibrations, 
we can affirmatively solve Conjecture \ref{conj-str} for compact K\"ahler surfaces 
and  (1) of Conjecture \ref{conj-str} 
in the case where a MRC fibration can be chosen to be a morphism
(see Corollary \ref{cor-main}). 
Further, by applying Theorem \ref{thm-main} to Albanese maps, 
we can obtain a vanishing theorem for the global holomorphic $1$-forms
(see Corollary \ref{cor-mainnnn}).
This vanishing theorem is an extension of \cite[Theorem 1.7]{Yan17a}.

\begin{theo}\label{thm-main}
Let $(X, g)$ be a compact K\"ahler manifold 
with semi-positive holomorphic sectional curvature  and  
let $Y$ be a compact K\"ahler manifold with pseudo-effective canonical bundle. 
Further let $\phi: X \to Y$ be a morphism from $X$ to $Y$. 
Then the following statements hold\vspace{0.1cm}\,$:$
\begin{itemize}
\item[(1)] $\phi$ is a smooth morphism $($that is, a submersion$)$. \vspace{0.2cm}
\item[(2)] The standard exact sequence of vector bundles obtained from $(1)$
$$
0 \xrightarrow{\quad \quad} T_{X/Y}:=\Ker d\phi_* \xrightarrow{\quad \quad}  
T_X \xrightarrow{\quad d\phi_* \quad}
\phi^{*} T_Y
\xrightarrow{\quad \quad} 0
$$ 
splits. 
Moreover its holomorphic splitting
$$
T_X = T_{X/Y} \oplus \phi^{*} T_Y
$$
coincides with the orthogonal decomposition of $T_{X}$ with respect to $g$. 
Here $T_X$ $($resp. $T_Y$$)$ is the $($holomorphic$)$ tangent bundle of $X$ 
$($resp. $Y$$)$.
\vspace{0.2cm}
\item[(3)] Let $g_Q$ be the hermitian metric on $\phi^{*} T_Y$ induced by 
the above exact sequence and the given metric $g$. 
Then there exists a hermitian metric  $g_Y$ on $T_Y$ with the following properties\,$:$ 
\vspace{0.1cm}
\begin{itemize}
\item[$\bullet$] $g_Q$ is obtained from the pull-back of $g_Y$ $($namely, $g_Q=\phi^* g_Y$$)$. 
\vspace{0.1cm}
\item[$\bullet$] The holomorphic sectional curvature of $(Y, g_Y)$ is identically zero. 
In particular, the image $Y$ is  flat and it admits a finite \'etale cover $T \to Y$
by a complex torus $T$. 
\end{itemize}
\vspace{0.2cm}
\item[(4)] $\phi$ is locally trivial if we further assume that $X$ is projective. 
\end{itemize}
\end{theo}

\begin{cor}\label{cor-main}
Let $X$ be a compact K\"ahler manifold with semi-positive holomorphic sectional curvature. 
Then the followings hold\,$:$
\vspace{0.2cm}\\
\quad $\bullet$ All the statements of Conjecture \ref{conj-str} hold  
in the case of $X$ being a surface. 
\vspace{0.1cm}\\
\quad $\bullet$ The  statement $(1)$ of Conjecture \ref{conj-str} holds 
if $X$ is projective and a MRC fibration of $X$ can be chosen to be a morphism. 
\end{cor}

\begin{cor}\label{cor-mainnnn}
Let $(X, g)$ be a compact K\"ahler manifold 
with semi-positive holomorphic sectional curvature. 
Then we have  
$$
h^{0}(X, \Omega_X)\leq \dim X - n_{{\rm{tf}}}{(X, g)}. 
$$ 
In particular, we obtain $h^{0}(X, \Omega_X)=0$ 
if $n_{{\rm{tf}}}{(X, g)}=\dim X$ $($which is satisfied if the holomorphic sectional curvature is quasi-positive$)$. 
\end{cor}

For the proof of Theorem \ref{thm-main}, 
we will carefully observe the curvature current and its integration of 
an induced \lq \lq singular" hermitian metric on $\phi^* K_Y$, 
and further we investigate the scalar curvature of the K\"ahler form $g$, 
based on the idea in \cite{HW15}. 
The main difficulty of Theorem \ref{thm-main} is that the given metric $g$ has no a priori relation with the morphism $\phi$. 
To overcome this difficulty, 
we will show 
that all the tangent vector in the horizontal direction of $\phi:X\to Y$ are truly flat, 
which produces a relation  (for example, the statements (2) and (3)) between the metric $g$ and the morphism $\phi$. 
The key point here is 
to construct a suitably chosen orthonormal basis of $T_X$ 
by using an argument on a partial positivity developed in \cite{Mat18}.

By modifying the above techniques for a general MRC fibration $\phi: X \dashrightarrow Y$ 
(which is not necessarily a morphism), 
we can prove that the numerical dimension of the image $Y$ is equal to zero. 
Moreover we can obtain the same conclusions as in (1), (2), (3) of Theorem \ref{thm-main} 
over the smooth locus of $\phi$. 

\begin{theo}\label{thm-mainn}
Let $(X, g)$ be a compact K\"ahler manifold 
with semi-positive holomorphic sectional curvature and  
$Y$ be a compact K\"ahler manifold with pseudo-effective canonical bundle. 
Let $\phi: X \dashrightarrow Y$ be an almost holomorphic map from $X$ to $Y$.

Then the numerical dimension $\nu(Y)$ of $Y$ is equal to zero. 
Moreover, the same conclusions as in $(1)$, $(2)$, $(3)$ of Theorem \ref{thm-main} hold 
if we replace $X$ and $Y$ in the statements of Theorem \ref{thm-main} 
with $X_1$ and $Y_1$ $($see Theorem \ref{r-thm-mainn} for the precise statement$)$. 
Here $X_1$ and $Y_1$ are Zariski open sets such that 
$\phi :X_1:=\phi^{-1}(Y_{1}) \to Y_1$ is a morphism. 
\end{theo}


In Section \ref{Sec-2}, we will recall some basic results on curvature and truly flat tangent vectors. 
In Section \ref{Sec-3}, we will prove all the theorems and corollaries. 
In Section \ref{Sec-4}, we will discuss open problems related to the geometry of semi-positive holomorphic sectional curvature. 

\subsection*{Acknowledgements}
The author wishes to thank Professor Xiaokui Yang 
for stimulating discussions on RC positivity and related open problems, 
and he also wishes to thank 
Professor Masaaki Murakami for helpful comments on the classifications of surfaces. 
He would like to thank the members of 
Morningside Center of Mathematics, Chinese Academy of Sciences 
for their hospitality during my stay. 
He is supported by the Grant-in-Aid 
for Young Scientists (A) $\sharp$17H04821 from JSPS.

\section{Preliminaries}\label{Sec-2}
For reader's convenience, we summarize some formulas and  properties of 
curvature tensors, holomorphic sectional curvature, and truly flat tangent vectors 
in this section.

\subsection{Curvature and exact sequences of vector bundles}\label{subsec-2-1}
In this subsection, we recall several formulas of curvature of induced hermitian metrics 
and properties of exact sequences of vector bundles.

Let $(E, g)$ be a (holomorphic) vector bundle on a complex manifold $X$ 
equipped with a (smooth) hermitian metric $g$. 
The Chern curvature of $(E, g)$  
$$
\sqrt{-1}\Theta_{g}:=\sqrt{-1}\Theta_{g}(E) 
\in C^{\infty}(X, \Lambda^{1,1}\otimes \End (E)), 
$$
defines the curvature tensor 
$$
R_g:=R_{(E,g)} \in 
C^{\infty}(X, \Lambda^{1,1}\otimes E^\vee \otimes \bar E^{\vee})
$$
to be   
$$
R_g(v, \bar w,e, \bar f):= 
\big \langle \sqrt{-1}\Theta_{g}(v, \bar w)(e), f \big \rangle_g
$$
for tangent vectors $v, w \in T_X$ and vectors $e, f \in E$. 
We denote the dual vector bundle of $E$ by the notation $E^\vee$ and 
the inner product with respect to $g$ by the notation $\langle \bullet, \bullet \rangle_g$ 
throughout this paper. 
The metric $g$ induces the hermitian metric $\Lambda^m g$ 
on the vector bundle $\Lambda^m E$ of the $m$-th exterior product. 
Then it follows that 
\begin{align}\label{eq-8}
&\big \langle \sqrt{-1}\Theta_{\Lambda^m g}(v, \bar v)
(e_{1}\wedge e_{2}\wedge \dots \wedge e_m), e_{1}\wedge e_{2}\wedge \dots \wedge e_m 
\big \rangle_{\Lambda^m g}\\ \notag
=&\sum_{k=1}^{m}\big \langle \sqrt{-1}\Theta_{g}(v, \bar v)(e_k), e_k \big \rangle_g 
\end{align}
for a tangent vector $v \in T_X$ and 
vectors $\{e_k\}_{k=1}^{m}$ in $E$ with $\langle e_{i}, e_{j} \rangle_g = \delta_{ij}$ 
since the curvature $\sqrt{-1}\Theta_{\Lambda^m g}$ associated to  $\Lambda^m g$ satisfies that 
$$
\sqrt{-1}\Theta_{\Lambda^m g}(v, \bar v)(e_{1}\wedge e_{2}\wedge \dots \wedge e_m)=
\sum_{k=1}^{m}  e_{1}\wedge \cdots  \wedge e_{k-1}\wedge \sqrt{-1}\Theta_{g}(v, \bar v)(e_{k}) \wedge e_{k+1} \wedge \cdots \wedge e_m. 
$$
In particular, the curvature $\sqrt{-1}\Theta_{\det g}$ of the determinant bundle 
$\det E:=\Lambda^{{\rm{rk}}E} E$ with the induced metric $\det g:=\Lambda^{{\rm{rk}}E} g$ satisfies that 
\begin{align*}
\sqrt{-1}\Theta_{\det g}(v, \bar v)
=&\big \langle \sqrt{-1}\Theta_{\det g}(v, \bar v)
(e_{1}\wedge e_{2}\wedge \dots \wedge e_{{{\rm{rk}}E}} ), e_{1}\wedge e_{2}\wedge \dots \wedge e_{{{\rm{rk}}E}} 
\big \rangle_{\det g}\\ \notag
=&\sum_{k=1}^{{\rm{rk}}E}\big \langle \sqrt{-1}\Theta_{g}(v, \bar v)(e_k), e_k \big \rangle_g 
\end{align*}
for an orthonormal basis $\{e_k\}_{k=1}^{{\rm{rk}}E}$ of $E$.

For a subbundle $S$ of $E$ and its quotient vector bundle $Q:=E/S$, 
we consider the hermitian metric $g_S$ (resp. $g_Q$) on $S$ (resp. $Q$) 
induced by the metric $g$ and the exact sequence 
$$
0 \xrightarrow{\quad \quad} (S, g_S) \xrightarrow{\quad \quad}  
(E, g) \xrightarrow{\quad  \quad}(Q, g_Q)
\xrightarrow{\quad \quad} 0. 
$$ 
The quotient bundle $Q$ is isomorphic to the orthogonal complement 
$S^{\bot}$ of $S$ in $(E, g)$ as $C^\infty$-vector bundles. 
By this isomorphism, 
the quotient bundle $Q$ can be identified with the $C^{\infty}$-vector bundle $S^{\bot}$ and 
the metric $g_Q$ can be regarded as the hermitian metric on $S^{\bot}$. 
Also, it can be proven that there exist smooth sections  (which are called fundamental forms) 
$$A\in C^{\infty}(X, \Lambda^{1,0}\otimes \Hom(S, S^{\bot})) 
\text { and }
B\in C^{\infty}(X, \Lambda^{0,1}\otimes \Hom(S^{\bot}, S))
$$ 
satisfying that 
\begin{align}
\big \langle \sqrt{-1}\Theta_{g}(v, \bar v)(e), e \big \rangle_g + 
\big \langle B_{\bar v}  (e), B_{\bar v}  (e) \big \rangle_{g_S} &= 
\big \langle \sqrt{-1}\Theta_{g_Q}(v, \bar v)(e), e \big \rangle_{g_Q}, \label{eq-1} \\
\big \langle \sqrt{-1}\Theta_{g}(v, \bar v)(f), f \big \rangle_g - 
\big \langle A_v (f), A_v (f) \big \rangle_{g_Q} &= 
\big \langle \sqrt{-1}\Theta_{g_S}(v, \bar v)(f), f \big \rangle_{g_S}, \notag\\
\big \langle A_v (f), e \big \rangle_{g_Q}+\big \langle f, B_{\bar v} (e) \big \rangle_{g_S} &=0 \notag 
\end{align}
for a tangent vector $v \in T_X$, a vector $e \in S^{\bot}$, and a vector $f \in S$. 
In particular, we have 
\begin{align}
\big \langle \sqrt{-1}\Theta_{g}(v, \bar v)(e), e \big \rangle_g  &\leq 
\big \langle \sqrt{-1}\Theta_{g_Q}(v, \bar v)(e), e \big \rangle_{g_Q}, \label{eq-1a} \\
\big \langle \sqrt{-1}\Theta_{g}(v, \bar v)(f), f \big \rangle_g  &\geq 
\big \langle \sqrt{-1}\Theta_{g_S}(v, \bar v)(f), f \big \rangle_{g_S}. \label{eq-2a}
\end{align}
Moreover it can be shown that the above exact sequence determines 
the holomorphic orthogonal decomposition  
$E=S \oplus Q$ (that is, $S^{\bot}$ is a holomorphic vector bundle and it is isomorphic to $Q$) 
if and only if $A$ (equivalently $B$) is  identically zero. 

In the rest of this subsection, 
we summarize the notion of singular hermitian metrics on a line bundle $L$ (see \cite{Dem} for more details). 
A hermitian metric $h$ on $L$ is said to be  
a {\textit{singular hermitian metric}}, if 
$\log  |e|_h$ is an $L^{1}_{\rm{loc}}$-function for any local frame $e$ of $L$. 
Then the curvature current $\sqrt{-1}\Theta_h$ of $(L,h)$ is defined by   
\begin{equation*}
\sqrt{-1} \Theta_{h}:=\sqrt{-1} \Theta_{h}(L) :=-\deldel \log  |e|^2_h 
\end{equation*}
in the sense of distributions. 
The singular hermitian metric $h$ is said to have
neat {\textit{analytic singularities}}, 
if there exists an ideal sheaf $\mathcal{I} \subset \mathcal{O}_{X}$ such that 
the function $ - \log  |e|^2_h$  can be locally written as 
\begin{equation*}
- \log  |e|^2_h = c \log \big( 
|f_{1}|^{2} + |f_{2}|^{2} + \cdots + |f_{k}|^{2}\big) +\text{smooth function}, 
\end{equation*}
where $c $ is a positive real number and 
$f_{1}, \dots, f_{k}$ are local generators of $\mathcal{I}$. 
We say that $h$ has {\textit{divisorial singularities}} 
when the ideal sheaf $\mathcal{I}$ is defined by an effective divisor. 
The dual singular hermitian metric $h^{\vee}$ 
on the dual line bundle $L^\vee$ can be defined to be  
$|e^\vee|_{h^{\vee}}:=|e|_{h}^{-1}$ for the dual local frame $e^\vee$. 
Further, for a morphism $f: Z \to X$, the singular hermitian metric $f^{*}h$ on the pull-back $f^*L$ can also be defined to be  
$|f^* e|_{f^{*}h}:=f^*(| e|_{h})$ for the local frame $f^* e$ of $f^*L$.
Then we have 
$$
\sqrt{-1}\Theta_{h}=-\sqrt{-1}\Theta_{h^\vee}=\deldel \log |e^\vee|^2_{h^{\vee}} \text{\quad  and \quad } 
f^* \sqrt{-1}\Theta_{ h}:=\sqrt{-1}\Theta_{f^* h}.
$$ 

\subsection{Holomorphic sectional curvature and truly flat tangent vectors}\label{subsec-2-2}
In this subsection, we summarize some properties of holomorphic sectional curvature and truly flat tangent vectors. 
For a hermitian metric $g$ on the (holomorphic) tangent bundle $T_X$, 
the holomorphic sectional curvature $H_g$ is defined to be 
$$
H_g([v]):=\frac{\Rur{g}{v}{v}}{|v|_g^{4}}=\frac{\big \langle \sqrt{-1}\Theta_{g}(v, \bar v)(v), v \big \rangle_g}{|v|_g^{4}}
$$
for a non-zero tangent vector $v \in T_X$. 
The holomorphic sectional curvature $H_g$ is said to be {\textit{positive}} 
(resp. {\textit{semi-positive}}) 
if $H_{g}([v]) > 0$  (resp. $H_{g}([v]) \geq 0$) holds 
for any non-zero tangent vector $v \in T_X$. 
Also $H_g$ is said to be {\textit{quasi-positive}} 
if it is semi-positive everywhere and it is positive at some point in $X$. 

In this paper, we handle  only the case of $g$ being a K\"ahler metric 
(that is, the associated $(1,1)$-form $\omega_g$ is $d$-closed). 
In this case, the following symmetries hold\,$:$ 
$$
R_{g}(e_{i}, \bar e_{j}, e_{k}, \bar e_{\ell})
=R_{g}(e_{k}, \bar e_{\ell}, e_{i}, \bar e_{j})
=R_{g}(e_{k}, \bar e_{j}, e_{i}, \bar e_{\ell}). 
$$
The above symmetries lead to the following lemmas.

\begin{lemm}[{\cite[Lemma 4.1]{Yan17c}, \cite[Lemma 2.2]{Mat18} cf. \cite{Bre}, \cite{BKT13}}]\label{lem-ineq}
Let $g$ be a K\"ahler metric of $X$ and $V$ be a subspace of $T_{X,p}$ at a point $p \in X$. 
If a unit vector $v \in V$ is a minimizer of 
the holomorphic sectional curvature $H_g$ on $V$, 
that it, it satisfies 
$$
H_g([v])=\min\{H_g([x])\, |\, {0 \not = x \in V}\}, 
$$
then we have 
$$
2\Rur{g}{v}{x} \geq (1+|\langle v, x \rangle_g|^2) \Rur{g}{v}{v} 
$$
for any unit vector $x \in V$.  
In particular, if the holomorphic sectional curvature $H_g$ is semi-positive, 
a minimizer $v$ of $H_g$ on $V$ satisfies that 
$$
\Rur{g}{v}{x} \geq 0
$$
for any tangent vector $x \in V$. 
\end{lemm}

The above lemma was proved in \cite[Lemma 4.1]{Yan17c} 
when the subspace $V$ in Lemma \ref{lem-ineq} coincides with the  tangent space $T_{X,p}$. 
It is easy to see that the same argument as in \cite[Lemma 4.1]{Yan17c} 
works even in the case of $V$ being a subspace of $T_{X,p}$, 
and thus we omit the proof of Lemma \ref{lem-ineq}. 
Note that we can always take the minimizer of $H_{g}$ 
on a given subspace $V$ of $T_{X,p}$ at a point $p \in X$, 
since the holomorphic sectional curvature can be regarded as a smooth function 
on the projective space bundle $\mathbb{P}(T_X ^\vee)$ 
(that is, the set of all complex lines $[v]$ in $T_X$) and 
$\mathbb{P}(V^\vee) \subset \mathbb{P}(T_{X,p} ^\vee)$ is compact. 

Now we define truly flat tangent vectors and 
the invariant $n_{{\rm{tf}}}{(X, g)}$ introduced in \cite{HLWZ}.
We remark that the invariant $n_{{\rm{tf}}}{(X, g)}$ 
was denoted by the different notation $r_{{\rm{tf}}}^-$ in \cite{HLWZ}.

\begin{defi}[Truly flat tangent vectors and the invariant $n_{{\rm{tf}}}{(X, g)}$]\label{def-flat}
Let $(X, g)$ be a K\"ahler manifold. 
\ \\
\quad $\bullet$ A tangent vector $v \in T_X$ at $p$ is said to be {\textit{truly flat}} with respect to $g$ 
if $v$ satisfies that 
$$\text{
$R_{g}(v, \bar x, y,\bar z)=0$ 
for any tangent vectors  $x, y, z \in T_{X,p}$. 
}
$$
\quad $\bullet$ We define the subspace $V_{{\rm{flat}},p}$ of $T_{X,p}$ at $p$  by 
$$
V_{{\rm{flat}},p}:=\{v \in T_{X,p} \, |\, v \text{ is a truly flat tangent vector in }T_{X,p} \}.  
$$ 
\quad $\bullet$ We define the invariants $n_{{\rm{tf}}}{(X, g)}_p$ and $n_{{\rm{tf}}}{(X, g)}$ by 
$$
n_{{\rm{tf}}}{(X, g)}_p:=\dim X -  \dim V_{{\rm{flat}},p} \text{\quad  and \quad}
n_{{\rm{tf}}}{(X, g)}:=\dim X - \inf_{p \in X} \dim V_{{\rm{flat}},p}. 
$$
\end{defi}
It is easy to see that the  invariant $n_{{\rm{tf}}}{(X, g)}_p$ is 
lower semi-continuous with respect to $p \in X$ in the classical topology. 
In particular, if we have the equality $n_{{\rm{tf}}}{(X, g)}_p=n_{{\rm{tf}}}{(X, g)}$  at $p$, 
the same equality holds on a neighborhood of $p$. 
The following lemma gives a characterization of truly flat tangent vectors 
in terms of holomorphic sectional curvature and bisectional curvature.

\begin{lemm}[cf. {\cite{HLWZ}}]\label{lemm-flat}
Let $g$ be a K\"ahler metric of $X$ with semi-positive holomorphic sectional curvature 
and $V$ be a subspace of $T_{X,p}$ at a point $p \in X$. 
If a tangent vector $v \in T_X$  satisfies  that 
$$
H_g([v])=0 \quad \text{ and } \quad \Rur{g}{v}{w}=0 \text{ for any tangent vector $w \in V$}, 
$$
then $v$ satisfies that 
$$
R_{g}(v, \bar x, y, \bar z) =0\text{ for any tangent vectors } x,y,z \in V. 
$$
In particular, if $v$ satisfies the above assumptions for any tangent vector $w \in T_{X,p}$, 
then $v$ is a truly flat tangent vector at $p$. 
\end{lemm}
\begin{proof}
When the holomorphic sectional curvature is semi-negative and the subspace $V$ coincides with 
the tangent space $T_{X,p}$, 
the same conclusion was proved in \cite[Lemma 2.1]{HLWZ}. 
For reader's convenience, we will give a sketch of the proof.

For an arbitrary complex number $re^{\sqrt{-1}\theta}$,  
we obtain that 
\begin{align*}
0 &\leq H([v+re^{\sqrt{-1}\theta}w])|v+re^{\sqrt{-1}\theta}w|_g^4\\
&=2 \Re(e^{\sqrt{-1}\theta}R_g(v,\bar v,v,\bar w) ) r^3 
+ 2 \Re(e^{\sqrt{-1}\theta}R_g(v,\bar w,w,\bar w) ) r + R_g(w,\bar w,w,\bar w) 
\end{align*}
from the assumptions $\Rur{g}{v}{v}=0$ and $\Rur{g}{v}{w}=0$. 
Here we used the symmetries obtained from K\"ahler metrics. 
If $R_g(v,\bar v,v,\bar w)$ is not zero, 
we have a contradiction by suitably choosing $\theta$ such that 
$\Re(e^{\sqrt{-1}\theta}R_g(v,\bar v,v,\bar w) ) <0$ and by taking a sufficiently large $r>0$. 
Hence we obtain $R_g(v,\bar v,v,\bar w)=0$. 
By repeating the same argument as above for $e^{\sqrt{-1}\theta}R_g(v,\bar w,w,\bar w) $, 
we can see that $R_g(v,\bar w,w,\bar w)=0$ for any tangent vector $w \in T_X$. 
Then we can easily check the desired equality by the standard polarization argument. 
\end{proof}

\section{Proof of the results}\label{Sec-3}

\subsection{Proof of Theorem \ref{thm-main}}\label{subsec-3-1}
In this subsection, we give a proof of Theorem \ref{thm-main}. 
The arguments in this subsection will be modified  
to handle  almost holomorphic maps in the proof of Theorem \ref{thm-mainn}. 
This subsection is the core of this paper.

\begin{theo}[=Theorem \ref{thm-main}]\label{r-thm-main}
Let $(X, g)$ be a compact K\"ahler manifold 
with semi-positive holomorphic sectional curvature  and  
let $Y$ be a compact K\"ahler manifold with pseudo-effective canonical bundle. 
Further let $\phi: X \to Y$ be a morphism from $X$ to $Y$. 
Then the following statements hold\vspace{0.1cm}\,$:$
\begin{itemize}
\item[(1)] $\phi$ is a smooth morphism $($that is, a submersion$)$. \vspace{0.2cm}
\item[(2)] The standard exact sequence of vector bundles obtained from $(1)$
$$
0 \xrightarrow{\quad \quad} T_{X/Y}:=\Ker d\phi_* \xrightarrow{\quad \quad}  
T_X \xrightarrow{\quad d\phi_* \quad}
\phi^{*} T_Y
\xrightarrow{\quad \quad} 0
$$ 
splits. 
Moreover its holomorphic splitting
$$
T_X = T_{X/Y} \oplus \phi^{*} T_Y
$$
coincides with the orthogonal decomposition of $T_{X}$ with respect to $g$. 
Here $T_X$ $($resp. $T_Y$$)$ is the $($holomorphic$)$ tangent bundle of $X$ 
$($resp. $Y$$)$.
\vspace{0.2cm}
\item[(3)] Let $g_Q$ be the hermitian metric on $\phi^{*} T_Y$ induced by 
the above exact sequence and the given metric $g$. 
Then there exists a hermitian metric  $g_Y$ on $T_Y$ with the following properties\,$:$ 
\vspace{0.1cm}
\begin{itemize}
\item[$\bullet$] $g_Q$ is obtained from the pull-back of $g_Y$ $($namely, $g_Q=\phi^* g_Y$$)$. 
\vspace{0.1cm}
\item[$\bullet$] The holomorphic sectional curvature of $(Y, g_Y)$ is identically zero. 
In particular, the image $Y$ is  flat and it admits a finite \'etale cover $T \to Y$
by a complex torus $T$. 
\end{itemize}
\vspace{0.2cm}
\item[(4)] $\phi$ is locally trivial if we further assume that $X$ is projective. 
\end{itemize}
\end{theo}
\begin{proof}
Throughout this proof, 
let $(X, g)$ be a compact K\"ahler manifold 
with the semi-positive holomorphic sectional curvature $H_g$ and 
let $\phi: X \to Y$ be a morphism (that is, a surjective holomorphic map) 
to  a compact K\"ahler manifold $Y$ with the pseudo-effective canonical bundle $K_Y$. 
For simplicity, we put $n:=\dim X$ and $m:= \dim Y$. 
We will divide the proof into five steps to refer later. 

\begin{step}[Singularities of induced singular hermitian metrics]\label{step1}

Our first purpose is to prove that $\phi$ is actually a smooth morphism. 
For this purpose, in this step, we first construct a possibly \lq \lq singular" hermitian metric $G$ on the line bundle $\phi^* K_Y^{\vee}$ 
from the given K\"ahler metric $g$ of $X$ such that the singularities of $G$ corresponds to the non-smooth locus of $\phi$. 
This enables us to  reduce our first purpose to observe the singularities of $G$. 
Moreover, in this step, we show  that the pull-back $\pi^* G$ has divisorial singularities and 
its curvature current can be decomposed into a smooth part and a divisorial part, 
after we take a suitable modification $\pi : \bar X \to X$.

Now we have the injective sheaf morphism 
\begin{equation*}
(\phi^{*} K_Y, H)
\xrightarrow{\quad d\phi^* \quad} (\Lambda^m \Omega_X, \Lambda^m h)
\end{equation*}
between the vector bundle $\Lambda^m \Omega_X:=\Lambda^m T_X^\vee$ of the $m$-th exterior product 
and the line bundle $\phi^{*} K_Y:=\phi^{*} \Lambda^m \Omega_Y$. 
We interchangeably use the words  \lq \lq line bundles" and \lq \lq invertible sheaves" 
(also \lq \lq vector bundles" and \lq \lq locally free sheaves") throughout this paper. 
Note that the above morphism is not a bundle morphism 
since the rank drops on the non-smooth locus of $\phi$, 
but it is an injective morphism as sheaf morphisms. 
Let $h$ be the dual hermitian metric of $g$ 
on the (holomorphic) cotangent bundle $\Omega_X=T_X^{\vee}$ and 
let $\Lambda^m h$ be the induced metric on $\Lambda^m \Omega_X$. 
Then, from the above morphism, we can construct a possibly singular hermitian metric $H$ on $\phi^{*} K_Y$ to be  
$$
|e|_H:=|d\phi^* (e)|_{\Lambda^m h} \text{ for a local frame $e$ of $\phi^{*} K_Y$}. 
$$
From now on, we mainly consider the dual singular hermitian metric $G:=H^{\vee}=H^{-1}$ on $\phi^{*} K_Y^{^\vee}$. 
For a local coordinate $(t_1, t_2, \dots, t_m)$ of $Y$, 
the $m$-form $dt:=dt_1\wedge dt_2 \wedge \cdots \wedge dt_m$ naturally determines 
the local frame of $\phi^{*} K_Y$, which we denote by the same notation $dt$. 
By the definitions of the curvature and the dual metric, 
the curvature current of $(\phi^{*} K_Y^{^\vee}, G)$ can be locally written as 
$$
\sqrt{-1}\Theta_G:=\sqrt{-1}\Theta_G(\phi^{*} K_Y^{^\vee})=
\deldel \log |\phi^* dt|^2_{\Lambda^m h}, 
$$
where $\phi^* dt$ is the pull-back of the $m$-form $dt$ by $\phi$. 
We remark that  the pull-back $\phi^* dt$  coincides with 
the image $d\phi^* (dt)$ of the section $dt$ of $\phi^*K_Y$ by $d\phi^{*}$. 

By the above expression, 
it can be shown that the singular locus of $G$ 
(that is, the polar set of the quasi-psh function $\log |\phi^* dt|_{\Lambda^m h}$) 
coincides with the non-smooth locus of $\phi$, 
since the zero locus of the section $\phi^* dt$ of 
$\Lambda^m \Omega_X$ is equal to  the non-smooth locus of $\phi$. 
Therefore it is sufficient  for our first purpose (that is, the proof of the smoothness of $\phi$) to prove 
that $G$ is actually a smooth hermitian metric. 

We take a (log) resolution $\pi: \bar X \to X$ of the degenerate ideal $\mathcal{I}$ of the above sheaf morphism. 
The degenerate ideal $\mathcal{I}$ is the ideal sheaf 
generated by the coefficients of $\phi^* dt$ with respect to local frames of $\Lambda^m \Omega_X$. 
Then we obtain the following claim\,$:$
\begin{claim}\label{claim-div} 
Let $Z$ be the non-smooth locus of $\phi$. 
Then the following statements hold\,$:$ 
\ \vspace{0.1cm}\\
\quad $\bullet$ $ \pi^{-1}(Z)  $ has codimension one. 
\vspace{0.1cm}\\
\quad $\bullet$ $\pi: \bar X \setminus  \pi^{-1}(Z)  \cong X \setminus Z $. 
\vspace{0.1cm}\\
\quad $\bullet$ $\phi^* G$ has divisorial singularities along $\pi^{-1}(Z) $. 
\vspace{0.1cm}\\ 
More precisely, the pull-back $\pi^{*}\sqrt{-1}\Theta_G$ of 
the curvature current $\sqrt{-1}\Theta_G$ can be written as 
\begin{align*}
\pi^{*}\sqrt{-1}\Theta_G:=\deldel \log \pi^{*} (|\phi^* dt|^2_{\Lambda^m h})=\gamma + [E], 
\end{align*}
where $\gamma$ is a smooth $(1,1)$-form on $\bar X$ and $[E]$ is the integration current 
defined by an effective divisor $E$. 
\end{claim}
\begin{proof}[Proof of Claim \ref{claim-div}]
The subvariety $Z$ coincides with the support of the cokernel $\mathcal{O}_X/\mathcal{I}$, 
and thus the first and second statements are obvious by the choice of $\pi$. 
However the third statement seems to be a subtle problem, 
since we do not know whether or not the metric $G$ itself has neat analytic singularities 
(see Remark \ref{rem-div} for more details).

To check the third statement, we fix an arbitrary point $p \in \bar X$. 
When $p$ is outside $\pi^{-1}(Z)$, 
the metric $G$ is smooth on a neighborhood of $\pi(p)$ 
since $Z$ also coincides with the zero locus of $\phi^* dt$. 
The third statement is obvious in this case, 
and thus we may assume that $p \in \pi^{-1}(Z)$. 

We take a local frame $\{s_{i}\}_{i=1}^{N}$  of $\Lambda^m \Omega_X$ on a neighborhood of $\pi(p)$. 
Here we put $N:=\binom{n}{m}$ for simplicity.  
The holomorphic $m$-form $\phi^* dt$ can be locally written as 
$$\phi^* dt=\sum_{i=1}^{N} f_i s_i \text{  on a neighborhood of $\pi(p)$ }$$ 
for some holomorphic functions $\{f_{i}\}_{i=1}^{N}$. 
The degenerate ideal $\mathcal{I}$ is generated by $\{f_i\}_{i=1}^{N}$ and 
$\pi^{-1}\mathcal{I}=\mathcal{I} \cdot \mathcal{O}_{\bar X}$ is 
the ideal sheaf associated to an effective divisor $E$. 
Let $t$ be a local holomorphic function such that $t$ determines the effective divisor $E$. 
Then it follows that  $g_i:=\pi^{*} f_{i}/t$ is a holomorphic function and the common zero locus $\cap_{i=1}^N g_i^{-1}(0)$ is empty 
from the choice of $\pi$.
Therefore a simple computation yields
$$
\log \pi^{*} (|\phi^* dt|^2_{\Lambda^m h})=\log |t|^2+\log \sum_{i, j=1}^{N} g_{i} 
\bar{g_{j}}\pi^{*}\langle s_i, s_j \rangle_{\Lambda^m h}. 
$$
It can be proven that the Levi form of the first term is equal to the integration current $[E]$ 
by the Poincar\'e-Lelong formula. 
On the other hand, it follows that the Levi form of the second term determines a smooth $(1,1)$-form $\gamma$, 
since it is easy to see that the function 
$$
\sum_{i, j=1}^{N} g_{i} \bar{g_{j}}\pi^{*}\langle s_i, s_j \rangle_{\Lambda^m h}
$$ 
is a non-vanishing smooth function. 
\end{proof}
\begin{rem}\label{rem-div}
$\bullet$ It follows that the smooth form $\gamma$ can be identified with the curvature $\sqrt{-1}\Theta_G$ 
under the isomorphism $\pi: \bar X \setminus \pi^{-1}(Z) \cong X \setminus Z$ from the second and third properties.
\vspace{0.1cm}\\
$\bullet$ The metric $G$ itself may not have neat analytic singularities 
although the pull-back $\pi^* G$ by $\pi$ has divisorial singularities. 
For example, in the case of $n=2$ and $m=1$, 
we consider the following situation\,$:$ 
\begin{align*}
\phi^{*} dt=z_1 s_1 +z_2 s_2 \text{ and } h=\left[
\begin{array}{cc}
2 & 1 \\
1 & 2 \\
\end{array}
\right]
\text{with respect to a local frame $(s_1, s_2)$ of $\Omega_X$}. 
\end{align*}
Here $(z_1, z_2)$ is a local coordinate of $X$. 
Then we can see that the function   
$$
\frac{|\phi^{*} dt|^2_h}{|z_1|^2+|z_2|^2}=
\frac{2|z_1|^2+\bar z_1 z_2 + z_1 \bar z_2 + 2|z_2|^2}{|z_1|^2+|z_2|^2}
$$
can not be extended to a smooth function defined at the origin. 
Of course, when we take a resolution of the degenerate ideal 
(which is just one point blow-up in this case), 
we can easily check that the pull-back of the above function is a non-vanishing smooth function. 
\end{rem}
\end{step}

\begin{step}[Construction of orthonormal basis in the horizontal direction]\label{step2}
In this step, by using the argument in \cite[Lemma 3.5]{Mat18}, 
we will choose a suitable orthonormal basis of $T_{X}$ 
at a smooth point $p$ of $\phi$,  
in order to obtain a partial positivity of $\sqrt{-1}\Theta_{G}$ and $\gamma$ in the horizontal direction.

For a given point $p \in X$ at which $\phi$ is smooth, 
we consider 
the standard exact sequence 
\begin{equation*}\label{eq-stan}
0 \xrightarrow{\quad \quad} T_{X/Y}:=\Ker d\phi_* \xrightarrow{\quad \quad}  
T_X \xrightarrow{\quad d\phi_* \quad}
\phi^{*} T_Y
\xrightarrow{\quad \quad} 0 \text{ at } p.
\end{equation*} 
In the proof, we say that a tangent vector $v \in T_X$ is in the {\textit{horizontal  direction}} (resp. in the {\textit{vertical direction}})
in the case of $v \in (T_{X/Y})^{\bot} $ (resp. $v\in T_{X/Y}$). 
Here $ (T_{X/Y})^{\bot}$ is the orthogonal complement of $T_{X/Y}$ in $T_X$ with respect to $g$ 
and it is identified with $\phi^{*}T_Y$ at $p$. 
Then we obtain the following claim\,$:$

\begin{claim}\label{claim-semi}
For a smooth point $p$ of $\phi$, 
there exists an orthonormal basis $\{e_k \}_{k=1}^{n}$ of $T_X$ at  $p$ with the following properties\,$:$
\ \vspace{0.1cm}\\
\quad $\bullet$ $\{e_i \}_{i=1}^{m}$ is an orthonormal basis of $(T_{X/Y})^{\bot}$ at $p$.  
\vspace{0.1cm}\\
\quad $\bullet$ 
$\Rur{g}{e_i}{e_j} \geq 0$ for any $1 \leq i, j \leq m$.  
\vspace{0.1cm}\\
\quad $\bullet$ $\sqrt{-1}\Theta_{G}(e_i, \bar e_i) \geq 0  \text{ for any $i=1,2,\dots, m$. }$
\end{claim}

\begin{proof}[Proof of Claim \ref{claim-semi}]
We first take an arbitrary orthonormal basis $\{e_k \}_{k=1}^{n}$ of $T_X$ at $p$ 
such that 
$$
(T_{X/Y})^{\bot}=\Spn\langle \{e_i \}_{i=1}^{m} \rangle
\quad \text{ and }  \quad 
T_{X/Y} = \Spn\langle \{e_j \}_{j=m+1}^{n} \rangle. 
$$
By choosing a new orthonormal basis $\{e_i \}_{i=1}^{m}$ of $(T_{X/Y})^{\bot}$, 
we may assume that $e_{1}$ is the minimizer of $H_{g}$ on 
$(T_{X/Y})^{\bot}=\Spn\langle \{e_k \}_{k=1}^{m} \rangle$, 
that is, the unit tangent vector $e_{1}$ satisfies that 
$$
H_g([e_1])=\min\{ H_g([v])\, |\, 0 \not = v \in \Spn \langle \{e_k \}_{k=1}^{m} \rangle \}. 
$$
After we fix the tangent vector $e_1$ chosen as above, 
we choose an orthonormal basis $\{e_i \}_{i=2}^{m}$ of 
$$
(T_{X/Y} \oplus \Spn \langle e_1 \rangle )^{\bot}=\Spn\langle \{e_k \}_{k=2}^{m} \rangle
$$ such that 
$e_{2}$ is the minimizer of $H_{g}$ on $\Spn\langle \{e_k \}_{k=2}^{m} \rangle$. 
By repeating this process, 
we can construct an orthonormal basis $\{e_i \}_{i=1}^{m}$ of $(T_{X/Y})^{\bot}$ 
satisfying that 
\begin{align*}
H_g([e_i])=\min\{H_g([v])\, |\, 0 \not = v \in \Spn \langle \{e_k \}_{k=i}^{m} \rangle\}. 
\end{align*}
for any $i = 1,2,\dots, m$. 

Then, for this orthonormal basis,  we can prove that 
$$\text{
$\Rur{g}{e_i}{e_j} \geq 0$ for any $1 \leq i, j \leq m$. 
}
$$  
Indeed, we may assume that $ i \leq j $
by the symmetry 
$$
\Rur{g}{e_i}{e_j}=\Rur{g}{e_j}{e_i}.
$$
Further, for $ i \leq j $, 
the tangent vector $e_i$ is the minimizer of $H_g$ on the subspace $\Spn \langle \{e_k \}_{k=i}^{m} \rangle$ which contains $e_j$. 
Therefore it follows that  $\Rur{g}{e_i}{e_j}$ is non-negative 
from Lemma \ref{lem-ineq}.

By applying the formulas (\ref{eq-8}) and  (\ref{eq-1a})  to 
the exact sequence 
\begin{equation*}\label{eq-surj}
0 \xrightarrow{\quad \quad} \Ker d\phi_*
\xrightarrow{\quad \quad}
(\Lambda^m T_X, \Lambda^m g)
\xrightarrow{\quad d\phi_* \quad} (\phi^{*} K_Y^{\vee}, G) \xrightarrow{\quad \quad} 0
\end{equation*}
on a neighborhood of $p$, 
we obtain that  
\begin{align}\label{eq-4}
\sum_{k=1}^{m}\Rur{g}{v}{e_k}&=  \big \langle \sqrt{-1}\Theta_{\Lambda^m g}(v, \bar v)
(e_1\wedge e_2 \wedge \dots \wedge e_m), e_1\wedge e_2 \wedge \dots \wedge e_m 
\big \rangle_{\Lambda^m g}  \\ \notag
&\leq \big \langle \sqrt{-1}\Theta_{G}(v, \bar v)
(e_1\wedge e_2 \wedge \dots \wedge e_m), e_1\wedge e_2 \wedge \dots \wedge e_m 
\big \rangle_{G}
\\
&=  \sqrt{-1}\Theta_{G}(v, \bar v)  |e_1\wedge e_2 \wedge \dots \wedge e_m|^2_{G} 
\notag \\&= \sqrt{-1}\Theta_{G}(v, \bar v) \notag
\end{align}
for any tangent vector $v \in T_X$.  
Note that $G$ (defined by the dual metric of $H$) is equal to 
the quotient metric induced by $\Lambda^m g$ since $p$ is a smooth point of $\phi$.
When we consider the above formula in the case of $v=e_{i}$, 
we can see that the left hand side is non-negative by the second statement in Claim \ref{claim-semi}. 
Therefore we can conclude that 
$\sqrt{-1}\Theta_{G}(e_i, \bar e_i)$ is non-negative
for any $i=1,2,\dots, m$. 
\end{proof}
\begin{rem}\label{rem-semi}
At the end of the proof, we can conclude 
that all the tangent vectors in the horizontal direction are actually truly flat and that 
the curvature $\sqrt{-1}\Theta_{G}$ is flat, 
but a further argument is needed for these conclusions. 
\end{rem}
\end{step}

\begin{step}[Positivity of scalar curvature and its integration]\label{step3}
In this step, we will consider the scalar curvature of $g$ and its integration, 
based on the idea in \cite{HW15}. 
Let $\omega$ be the K\"ahler form associated to the K\"ahler metric $g$. 
The first Chern class of $\pi^*\phi^{*} K^\vee_Y$ can be represented 
by the curvature current $\pi^{*} \sqrt{-1}\Theta_{G}/2\pi$. 
Hence, by taking the wedge product of the equality in Claim \ref{claim-div} 
with $ \pi^* \omega^{n-1}$ and by considering the integration over $\bar X$, 
we obtain 
\begin{align}\label{eq-5}
2 \pi \int_{\bar X} c_1(\pi^*\phi^{*} K^\vee_Y) \wedge \pi^* \omega^{n-1}
=\int_{\bar X} \gamma \wedge \pi^* \omega^{n-1} + \int_E \pi^* \omega^{n-1}. 
\end{align}
The purpose of this step is to prove that the first term of the right hand side is non-negative. 
If it is proven, all the terms can be shown to be  zero. 
Indeed, the left hand side is non-positive since $K_Y$ is pseudo-effective by the assumption and 
the second term of the right hand side is non-negative. 
We will show that this observation leads to a certain flatness of $K_Y$ and the smoothness of $\phi$ in Step \ref{step4}.

We first decompose the first term into the vertical part and the horizontal part. 
The integration of  $\gamma \wedge \pi^* \omega^{n-1}$ on $\bar X$ is equal to the integration on a Zariski open set
since $\gamma$ and $\pi^* \omega$ are smooth differential forms. 
Further $\bar X \setminus \pi^{-1}(Z)$ is isomorphic to  $X \setminus Z$ by the morphism $\pi$ 
and the equality $\gamma=\pi^{*}\sqrt{-1}\Theta_G$ holds on  the Zariski open set $\bar X \setminus \pi^{-1}(Z) $ 
(cf. Remark \ref{rem-div}). 
Therefore we can obtain that  
\begin{align*}
\int_{\bar X} \gamma \wedge \pi^* \omega^{n-1}&=
\int_{\bar X \setminus \pi^{-1}(Z)} \gamma \wedge \pi^* \omega^{n-1}\\
&=\int_{\bar X \setminus \pi^{-1}(Z)} \pi^{*}(\sqrt{-1}\Theta_G \wedge  \omega^{n-1})\\
&=\int_{X \setminus Z} \sqrt{-1}\Theta_G \wedge  \omega^{n-1}\\
&=\int_{X_0} \sqrt{-1}\Theta_G \wedge  \omega^{n-1}. 
\end{align*}
Here $X_0$ is  the Zariski open set defined by $X_0:=\phi^{-1}(Y_0)$ and 
$Y_{0}$ is the maximal Zariski open set of $Y$ such that 
the restriction $\phi : X_{0}=\phi^{-1}(Y_{0}) \to Y_{0}$ is a smooth morphism over $Y_0$.

On the other hand,  for a given point $p \in X_{0}$, 
we take an orthonormal basis $\{e_{k}\}_{k=1}^n$ 
of $T_X$ at $p$ satisfying the properties in Claim \ref{claim-semi}. 
Then we have  $\omega=  (\sqrt{-1}/2)\sum_{k=1}^{n}e_k^{\vee} \wedge \bar e_k^{\vee}$ at $p$, 
and thus we obtain  
\begin{align}\label{eq-7}
&\frac{n}{2}\int_{X_0} \sqrt{-1}\Theta_G \wedge  \omega^{n-1}\\
=&\int_{X_0} \sum_{i=1}^{m} \sqrt{-1}\Theta_{G}(e_i, \bar e_i) \, \omega^{n} +
\int_{X_0} \sum_{j=m+1}^{n} \sqrt{-1}\Theta_{G}(e_j, \bar e_j)\, \omega^{n} \notag
\end{align}
from straightforward computations of the scalar curvature. 
The integrand of the first term (which measures positivity of the scalar curvature in the horizontal direction) is non-negative by Claim \ref{claim-semi}. 
We will show that the second term (that is, the vertical part) 
is also non-negative 
by using Stokes's theorem and Fubini's theorem (see Claim \ref{claim-key}). 
Note that the integrand of the second term can be shown to be non-negative 
later (cf. Remark \ref{rem-semi}).  
However it seems to be quite difficult to directly check this fact. 
For this reason, we need to handle  the integration instead of the integrand.

\begin{claim}\label{claim-key}
Under the above situation, the second term is non-negative, namely, we have 
$$\int_{X_0} \sum_{j=m+1}^{n} \sqrt{-1}\Theta_{G}(e_j, \bar e_j)\, \omega^{n} \geq 0.
$$
\end{claim}

\begin{proof}[Proof of Claim \ref{claim-key}]
Let $\omega_Y$ be a K\"ahler form on $Y$. 
Then, for a given local coordinate $(t_1, t_2, \dots, t_m)$ of $Y_0$, 
there exists a smooth positive function $f$ defined on an open set in $Y_0$ such that 
$$
\omega^n= \frac{1}{\phi^{*}f \cdot |\phi^* dt|^2_{\Lambda^m h}} \, 
\phi^{*} \omega_Y^{m}  \wedge \omega^{n-m}, 
$$ 
where $dt:=dt_{1}\wedge dt_2 \wedge \cdots \wedge dt_m$. 
We remark that $\phi^{*}f$ and $|\phi^* dt|^2_{\Lambda^m h}$ depend on the choice of local coordinates, 
but the product is independent of the coordinates and it is globally defined on $X_0$. 
Indeed, it can be seen that    
\begin{align*}
\langle \phi^* dt_{\ell}, e_j^{\vee}  \rangle_h =\langle \phi^*dt_{\ell}, e_j  \rangle_{\rm{pairing}}=
\langle dt_{\ell}, \phi_*e_j  \rangle_{\rm{pairing}}=0 
\end{align*} 
for any $j=m+1,\dots, n$ since $e_j$ is in the kernel of $d\phi_*$.  
Therefore we obtain 
\begin{align*}
\phi^* dt_\ell&=\sum_{k=1}^{n}\langle \phi^* dt_{\ell}, e_k^{\vee}  \rangle_h \, e_{k}^{\vee}
=\sum_{i=1}^{m}\langle \phi^* dt_{\ell}, e_i^{\vee}  \rangle_h \, e_{i}^{\vee}. 
\end{align*}
Further we obtain \begin{align*}
\phi^* dt &= \det [\langle \phi^{*}dt_{\ell}, e_i^{\vee}  \rangle_h]\, e_1^\vee \wedge e_2^\vee \wedge \cdots  \wedge e_m^\vee 
\text{ \quad and \quad}|\phi^* dt|^2_{\Lambda^m h}=\big| \det [\langle \phi^* dt_{\ell}, e_i^{\vee}  \rangle_h  ] \big|^2
\end{align*}
by straightforward computations. 
On the other hand, the K\"ahler form $\omega_Y$ can be locally written as 
$$
\omega_Y=\sqrt{-1}\sum_{i,j=1}^{m} f_{ij} dt_i \wedge d\bar t_{j} 
$$
in terms of the given local coordinate $(t_1, t_2, \dots, t_m)$. 
From this local expression, we can easily show that 
\begin{align*}
\phi^{*} \omega_Y^m \wedge \omega^{n-m} &=c_{n, m}\, \phi^{*}(\det[f_{ij}])\, \phi^{*}(dt \wedge \bar dt)\wedge \omega^{n-m}\\
&=d_{n, m}\, \phi^{*}(\det[f_{ij}])\, \big| \det [\langle \phi^{*}dt_{\ell}, e_i^{\vee}  \rangle_h  ] \big|^2 \,\omega^{n}, 
\end{align*}
where $c_{n,m}$ and $d_{n,m}$ are the universal constants depending only on $n$ and $m$. 
Therefore it can be seen that $f:=d_{n,m} \det[f_{ij}]$ satisfies the desired equality.

By Fubini's theorem, 
we have 
\begin{align*}
\int_{X_0} \sum_{j=m+1}^{n} \sqrt{-1}\Theta_{G}(e_j, \bar e_j)\, \omega^{n}&=
\int_{Y_0} \frac{1}{f} \,  \omega_Y^m \, \int_{X_y}  \frac{1}{|\phi^* dt|^2_{\Lambda^m h}}
\sum_{j=m+1}^{n}  \sqrt{-1}\Theta_{G}(e_j, \bar e_j) \, \omega^{n-m}\\
&=\frac{n-m}{2}\int_{Y_0} \frac{1}{f} \,  \omega_Y^m \,  \int_{X_y}  \frac{1}{|\phi^* dt|^2_{\Lambda^m h}}
\sqrt{-1}\Theta_{G} \wedge \omega^{n-m-1},  
\end{align*}
where $X_y$ is the fiber of $\phi$ at $y \in Y_0$. 
Here we used the equality 
$$
\sum_{j=m+1}^{n}  \sqrt{-1}\Theta_{G}(e_j, \bar e_j) \, \omega^{n-m}
= \frac{n-m}{2}
\sqrt{-1}\Theta_{G} \wedge \omega^{n-m-1}  
$$
of the scalar curvature on the fiber $X_y$. 
We finally prove that the fiber integral in the above equality is non-negative. 
For simplicity, we put $F:=|\phi^* dt|^2_{\Lambda^m h}$. 
Then, by the definition of the curvature $\sqrt{-1}\Theta_{G}$, 
we can show that 
\begin{align*}
&\int_{X_y}  \frac{1}{|\phi^* dt|^2_{\Lambda^m h}}\sqrt{-1}\Theta_{G} \wedge \omega^{n-m-1}\\
=&
\int_{X_y}  \frac{1}{F} \deldel \log F 
\wedge \omega^{n-m-1}\\
=&\sqrt{-1} \int_{X_y}  \partial \Big( 
\frac{1}{F}\, \overline{\partial} \log F \wedge \omega^{n-m-1} \Big)
-\sqrt{-1} \int_{X_y}  \partial \Big( \frac{1}{F} \Big) \wedge \overline{\partial} 
\log F \wedge \omega^{n-m-1}\\
=& 
\int_{X_y} \frac{1}{F^3}  \sqrt{-1}\partial F \wedge \overline{\partial} F \wedge \omega^{n-m-1}.  
\end{align*}
The last equality follows from Stokes's theorem. 
The integrand of the right hand side is non-negative, 
and thus the desired inequality can be obtained. 
\end{proof}
\end{step}

\begin{step}[Curvature of the canonical bundle $K_Y$]\label{step4}
In this step, from the assumption that $K_Y$ is pseudo-effective,  
we will show that the curvature $\sqrt{-1}\Theta_G$ is flat and 
$\phi$ is a smooth morphism. 
The key point here is 
the observation on the flatness of curvature in the horizontal direction 
obtained from Claim \ref{claim-semi} and Claim \ref{claim-key}.

\begin{claim}\label{claim-flat}
The following statements hold\,$:$
\vspace{0.1cm} \\
\quad $\bullet$ The canonical bundle $K_Y$ is numerically zero $($that is, $c_{1}(K_Y)=0$$)$.
\vspace{0.1cm} \\
\quad $\bullet$  $H_g([e_i])=\Rur{g}{e_i}{e_i}=0$ for any $i=1,2,\dots, m$. 
\vspace{0.1cm} \\
\quad $\bullet$  $\Rur{g}{v}{e_i}\geq 0$ for any tangent vector $v \in T_X$ and any $i=1,2,\dots, m$. 
\vspace{0.1cm} \\
\quad $\bullet$ The curvature $\sqrt{-1}\Theta_G$ is flat on $X$. 
In particular, the effective divisor $E$ is actually the zero divisor, 
and thus the morphism $\phi$ is smooth. 
\end{claim}
\begin{proof}[Proof of Claim \ref{claim-flat}]
The left hand side of the equality (\ref{eq-5}) is non-positive 
since $\pi^*\phi^{*} K_Y$ is pseudo-effective by the assumption, 
and further each term of the right hand side is non-negative 
by Claim \ref{claim-key}. 
Hence we obtain 
\begin{align*}
\int_{ X} c_1(\phi^{*} K^\vee_Y) \wedge \omega^{n-1}&=\int_{\bar X}  c_1(\pi^*\phi^{*} K^\vee_Y) \wedge \pi^* \omega^{n-1} =0. 
\end{align*}
In general,  
if a pseudo-effective line bundle $L$ satisfies $c_1(L) \cdot \{\omega^{n-1} \}=0$ for some K\"ahler form $\omega$, 
then $L$ is numerically zero (for example see \cite{Mat13}). 
Indeed, for an arbitrary $d$-closed $(n-1, n-1)$ form $\eta$, 
we can take a positive constant $C$ such that 
$$
\frac{1}{C}\, \omega^{n-1} \leq \eta \leq C \omega^{n-1}.  
$$ 
Then we obtain $c_1(L) \cdot \{\eta\}=0$ by the assumption $c_1(L) \cdot \{\omega^{n-1} \}=0$. 
This leads to $c_{1}(L)=0$ by the duality. 
Therefore it can be seen that $\phi^{*} K_Y$ is numerically zero.

On the other hand, by the equalities (\ref{eq-5}) and (\ref{eq-7}), 
we have  
$$
\int_{X_0} \sum_{i=1}^{m} \sqrt{-1}\Theta_{G}(e_i, \bar e_i) \, \omega^{n}=0. 
$$
It follows that $\sqrt{-1}\Theta_{G}(e_i, \bar e_i)=0$ 
for any $i=1,2,\dots, m$ at a  point $p \in X_0$
since  the integrand $\sqrt{-1}\Theta_{G}(e_i, \bar e_i) $ is 
non-negative by Claim \ref{claim-semi}.
By applying the formula (\ref{eq-4}) to the case of $v=e_i$, 
we obtain 
\begin{align*}
0 \leq \sum_{k=1}^{m}\Rur{g}{e_i}{e_k} \leq  \sqrt{-1}\Theta_{G}(e_i, \bar e_i)=0. 
\end{align*}
The left inequality follows from Claim \ref{claim-semi}. 
In particular, we can see that 
$$ \text{
$H_g([e_i])=\Rur{g}{e_i}{e_i}=0$ for any $i=1,2,\dots, m$. 
}
$$ 
This implies that $e_i$ is the minimizer of 
the semi-positive holomorphic sectional curvature $H_g$ on $T_X$, 
and thus it can be shown that $\Rur{g}{v}{e_i}$ is non-negative 
for any tangent vector $v \in T_X$
by Lemma \ref{lem-ineq}. 
By applying the formula (\ref{eq-4}) to an arbitrary tangent vector $v \in T_X$ again, 
we  obtain 
\begin{align*}
0 \leq \sum_{i=1}^{m}\Rur{g}{v}{e_i} \leq \sqrt{-1}\Theta_{G}(v, \bar v). 
\end{align*}

This means that the curvature $\sqrt{-1}\Theta_{G}$ is semi-positive on $X_0$. 
We can see that $\gamma \geq 0$ holds on the Zariski open set $\pi^{-1}(X_0)$, 
since $\gamma=\pi^{*}\sqrt{-1}\Theta_{G}$ holds on $\bar X \setminus \pi^{-1}(Z)$ 
and we have $X_0 \subset X \setminus Z$. 
Hence it follows that $\gamma \geq 0 $ on the ambient space $\bar X$ since $\gamma $ is a smooth form. 
By the above arguments, the first Chern class $c_{1}(\pi^* \phi^* K_Y^\vee)$ 
(which is numerically zero) is represented by 
the sum of the semi-positive form $\gamma$ and the positive current $[E]$. 
Therefore we can conclude that $\gamma=0$ and $E=0$ (namely, $\sqrt{-1}\Theta_{G}=0$). 
In particular, we can see that $G$ is a smooth metric (that is, $\phi $ is a smooth morphism).
\end{proof}
\end{step}

\begin{step}[Truly flatness in the horizontal direction and its applications]\label{step5}
In this step, we first show that the statement (2)  in Theorem \ref{thm-main} holds and 
all the tangent vectors in the horizontal direction are truly flat. 
We will prove the statement (3) as an application of the truly flatness. 
Further we finally obtain the statement (4) from the theory of foliations.

Now we have the exact sequence of vector bundles 
\begin{equation}\label{eq-restan}
0 \xrightarrow{\quad \quad} (T_{X/Y}, g_S) \xrightarrow{\quad \quad}
(T_X, g) \xrightarrow{\quad d\phi_* \quad}
(\phi^{*} T_Y, g_Q)
\xrightarrow{\quad \quad} 0
\end{equation} 
on the ambient space $X$ since $\phi$ is a smooth morphism by Claim \ref{claim-flat}.  
Let $g_Q$ (resp. $g_S$) be the induced hermitian metric on $\phi^{*} T_Y$ (resp. $T_{X/Y}$). 
Then we prove the following claim\,$:$

\begin{claim}\label{claim-truly}
The following statements hold\,$:$
\vspace{0.1cm}\\
\quad $\bullet$ $e_{i}$ is a truly flat tangent vector for any $i=1,2,\dots, m$. 
\vspace{0.1cm}\\
\quad $\bullet$ The exact sequence $($\ref{eq-restan}$)$ splits, 
and its splitting $T_X= T_{X/Y} \oplus \phi^{*} T_Y$ coincides with the orthogonal decomposition of $(T_X, g)$.  
\vspace{0.1cm}\\
\quad $\bullet$ There exists a hermitian metric $g_Y$ on $T_Y$ 
such that $g_Q=\phi^{*} g_Y$ and $H_{g_Y}\equiv 0$. 
In particular, the image $Y$ admits a finite \'etale cover $T \to Y$ by a complex torus $T$. 
\end{claim}

\begin{proof}[Proof of Claim \ref{claim-truly}]
For any $i=1,2,\dots, m$, the tangent vector $e_i$ is the minimizer of the holomorphic sectional curvature $H_g$, 
and further $\langle \sqrt{-1}\Theta_{g}(v, \bar v)(e_i), e_i \rangle_g $ is non-negative 
for any tangent vector $v \in T_X$ by Claim \ref{claim-flat}.  
By the formula (\ref{eq-1}), we obtain that 
\begin{align}\label{eq-6}
0 \leq \big \langle \sqrt{-1}\Theta_{g}(v, \bar v)(e_i), e_i \big \rangle_g + 
\big \langle B_{\bar v} (e_i), B_{\bar v} (e_i) \big \rangle_{g_S}
= \big \langle \sqrt{-1}\Theta_{g_Q}(v, \bar v)(e_i), e_i \big \rangle_{g_Q}
\end{align}
for a tangent vector $v \in T_X$. 

On the other hand, the induced metric $\det g_Q$ on $\phi^{*}K_Y^{\vee}=\det \phi^{*} T_Y$ 
coincides with the metric $G$ constructed in Step \ref{step1} 
and the curvature of $\det g_Q=G$ is flat by Claim \ref{claim-flat}. 
Therefore we obtain that  
\begin{align*}
&\sum_{i=1}^{m}\big \langle \sqrt{-1}\Theta_{g_Q}(v, \bar v)(e_i), e_i \big \rangle_{g_Q}\\
=&\big \langle \sqrt{-1}\Theta_{\det g_Q}(v, \bar v)
(e_1\wedge e_2 \wedge \cdots \wedge e_m ), (e_1\wedge e_2 \wedge \cdots \wedge e_m)
\big \rangle_{\det g_Q} \\
=&0
\end{align*}
by the equality (\ref{eq-8}). 
By combining with the inequality (\ref{eq-6}), 
we can obtain that 
\begin{align*}
\Rur{g}{v}{e_i} = \big \langle \sqrt{-1}\Theta_{g}(v, \bar v)(e_i), e_i \big \rangle_g =0 
\quad \text{ and } \quad  
\big \langle B_{\bar v} (e_i), B_{\bar v} (e_i) \big \rangle_{g_S} =0
\end{align*}
for any tangent vector $v \in T_X$ and $i=1,2,\dots, m$. 
Here we used the fact that 
$\langle B_{\bar v} (e_i), B_{\bar v} (e_i)  \rangle_{g_S}$ is non-negative. 
Then, by Lemma \ref{lemm-flat}, we can see that $e_i$ is truly flat 
since $e_i$ satisfies that $\Rur{g}{v}{e_i}=0 $ and $H_g([e_i])=0$. 
Further it follows that 
$B\in C^{\infty}(X, \Lambda^{0,1}\otimes \Hom(S^{\bot}, S))$ is identically zero 
since $\langle B_{\bar v} (\bullet), B_{\bar v} (\bullet)  \rangle_{g_S}$ 
is a semi-positive definite quadratic form on $S^{\bot}$ and 
its trace 
$\sum_{i=1}^{m}\langle B_{\bar v} (e_i), B_{\bar v} (e_i)  \rangle_{g_S}$ is zero 
by the above argument. 
Hence we obtain the holomorphic orthogonal decomposition $T_X= T_{X/Y} \oplus \phi^{*} T_Y$ (see subsection \ref{subsec-2-1}). 

Now we prove the last statement. 
For a (local) vector field $v $ of $ T_{Y}$ defined on an open set $U$ in $Y$, 
we consider 
the section $\phi^{*}v \in H^{0}(\phi^{-1}(U), \phi^{*}T_Y)$ 
defined by 
$$
\phi^{-1}(U) \ni p \to v(\phi(p)) \in T_{Y, \phi(p)}={(\phi^{*}T_Y)}_p,
$$ 
which we will denote by the notation $\phi^{*}v$. 
If the function $|\phi^{*} v|_{g_Q}$ is a constant on the fiber $X_y$, 
we can define the hermitian metric $g_Y$ of $Y$ by $|v|_{g_Y}:=|\phi^{*} v|_{g_Q}$. 
Then we have $g=\phi^{*} g_Y$ by the definition. 

If we can show that the restriction of $\deldel \log |\phi^{*} v|_{g_Q} $ 
to the fiber is a semi-positive $(1,1)$-form, 
it should be a constant by the maximal principle, 
since $|\phi^{*} v|_{g_Q}$ is a psh function globally defined on the compact fiber. 
For this purpose, we consider the sub-line bundle $L$ of $\phi^{*}T_Y$
$$
(L:=\Spn \langle \phi^{*} v \rangle, g_L) \subset (\phi^{*}T_Y, g_Q)
$$
spanned by $\phi^{*} v$. 
Let $g_L$ be the induced metric on $L$. 
By the definition of the curvature and the induced metric, 
we obtain 
$$
\sqrt{-1}\Theta_{g_L}: = \sqrt{-1}\Theta_{g_L}(L)=-\deldel \log |\phi^{*} v|^2_{g_L}=-\deldel \log |\phi^{*} v|^2_{g_Q}. 
$$
By applying the formula (\ref{eq-2a}) to the above injective bundle morphism, we obtain that  
$$
\sqrt{-1}\Theta_{g_L}(w, \bar w) |\phi^{*} v|^2_{g_L}
=\big \langle \sqrt{-1}\Theta_{g_L}(w, \bar w)(\phi^{*} v), \phi^{*} v \big \rangle_{g_L}
\leq \big \langle \sqrt{-1}\Theta_{g_Q}(w, \bar w)(\phi^{*} v), \phi^{*} v \big \rangle_{g_Q}
$$
for a tangent vector $w \in T_X$. 
We have already shown that the tangent vectors $\{e_{i}\}_{i=1}^{m}$ are truly flat by the above argument. 
The vector $\phi^{*} v$ can be written as a linear combination of $\{e_{i}\}_{i=1}^{m}$, 
and thus it is also truly flat. 
On the other hand, by the holomorphic orthogonal decomposition $T_X=T_{X/Y}\oplus \phi^{*} T_Y$, 
the section $\phi^{*} v$ of $\phi^{*} T_Y$ determines the section of $T_X$, 
which we denote by the same notation $\phi^{*} v$. 
Then we obtain 
$$
\big \langle \sqrt{-1}\Theta_{g_Q}(w, \bar w)(\phi^{*} v), \phi^{*} v \big \rangle_{g_Q}
=
\big \langle \sqrt{-1}\Theta_{g}(w, \bar w)(\phi^{*} v), \phi^{*} v \big \rangle_{g}  
=
0. 
$$
The right equality follows from the truly flatness of $\phi^{*} v$. 
Therefore we can see that $\sqrt{-1}\Theta_{g_L}$ is semi-negative (in particular $|\phi^{*} v|_{g_Q}$ is a constant).

We finally check that the holomorphic sectional curvature of $g_Y$ is identically zero. 
Note that, in general, a compact K\"ahler manifold is a complex torus up to finite \'etale covers  
when the holomorphic sectional curvature is identically zero 
(see \cite{Igu54}, \cite[Proposition 2.2]{HLW16}, \cite{Ber66}, \cite{Igu54}).
For a given tangent vector $v \in T_Y$, 
the vector $\phi^* v$ satisfies $d\phi_* (\phi^* v)=v$. 
Hence we can easily see that  
\begin{align*}
0=\big \langle \sqrt{-1}\Theta_{g}(\phi^* v, \bar \phi^* v)(\phi^{*} v), 
\phi^{*} v \big \rangle_{g}  
&=\big \langle \sqrt{-1}\Theta_{g_Q}(\phi^* v, \bar \phi^* v)(\phi^{*} v), \phi^{*} v \big \rangle_{g_Q}\\ 
&=\big \langle \sqrt{-1}\Theta_{g_Y}(v, \bar v)(v), v \big \rangle_{g_Y}
\end{align*}
by $d\phi_* (\phi^* v)=v$, $g_Q=\phi^{*}g_Y$, and the truly flatness of $\phi^{*} v$. 
\end{proof}

We check the statement (4) in Theorem \ref{thm-main}. 
When $X$ is projective, 
it can be shown that the morphism $\phi: X \to Y$ is a (holomorphic) fiber bundle 
(in particular, all the fibers are isomorphic)
by the classical Ehresmann theorem and \cite[Lemma 3.19]{Hor07}. 
\end{step}
We finish the proof of Theorem \ref{thm-main}. 
\end{proof}

\begin{rem}\label{rem-main}
If  the foliation $\phi^*T_Y \subset T_X$ obtained from Theorem \ref{thm-main} is integrable 
(that is, it is closed under the Lie bracket), 
then it can be shown that $\phi$ is locally trivial and 
we have the decomposition 
$$
X_{\rm{univ}} \cong Y_{\rm{univ}} \times F_{\rm{univ}}=\mathbb{C}^m \times F_{\rm{univ}}
$$
by the Ehresmann theorem (for example see \cite[Theorem 3.17]{Hor07}). 
The integrability of $\phi^*T_Y \subset T_X$ is satisfied 
when the dimension of $Y$ is one. 
In this case, the image $Y$ is automatically an elliptic curve by the statement (3) in Theorem \ref{thm-main}. 
See \cite{Hor07} and references therein for more details. 
\end{rem}

\subsection{Proof of Theorem \ref{thm-mainn}}\label{subsec-3-2}
In this subsection, we will prove Theorem \ref{thm-mainn} 
by modifying the arguments in the proof of Theorem \ref{thm-main} for almost holomorphic maps. 

\begin{theo}[=Theorem \ref{thm-mainn}]\label{r-thm-mainn}
Let $(X, g)$ be a compact K\"ahler manifold 
with semi-positive holomorphic sectional curvature and  
$\phi: X \dashrightarrow Y$ be an almost holomorphic map to 
a compact K\"ahler manifold $Y$ with pseudo-effective canonical bundle.
Let $X_1$ and $Y_1$ be Zariski open sets such that 
$\phi :X_1:=\phi^{-1}(Y_{1}) \to Y_1$ is a morphism. 
Then we have the followings\,$:$
\begin{itemize}
\item[(0)] The numerical dimension $\nu(Y)$ of $Y$ is equal to zero. 
\vspace{0.1cm}
\item[(1)] $\phi$ is a smooth morphism on $X_1$. 
\vspace{0.1cm}
\item[(2)] The standard exact sequence of vector bundles on $X_1$ 
$$
0 \xrightarrow{\quad \quad} T_{{X_1}/{Y_1}}:=\Ker d\phi_* \xrightarrow{\quad \quad}  
T_{X_1} \xrightarrow{\quad d\phi_* \quad}
\phi^{*} T_{Y_1}
\xrightarrow{\quad \quad} 0
$$ 
gives the holomorphic orthogonal decomposition  
$$
T_{X_1} = T_{X_1/Y_1} \oplus \phi^{*} T_{Y_1}. 
$$
Moreover, all the tangent vectors in $\phi^{*} T_{Y_1} \subset T_{X_1}$ are truly flat. 
\vspace{0.1cm}
\item[(3)] Let $g_Q$ be the hermitian metric on $\phi^{*} T_{Y_1}$ induced by 
the above exact sequence and the given metric $g$. 
Then there exists a hermitian metric  $g_Y$ on $T_{Y_1}$ with the following properties\,$:$ 
\vspace{0.1cm}
\begin{itemize}
\item[$\bullet$] $g_Q$ is obtained from the pull-back of $g_Y$ $($namely, $g_Q=\phi^* g_Y$$)$. 
\vspace{0.1cm}
\item[$\bullet$] The holomorphic sectional curvature of $(Y_1, g_Y)$ is identically zero. 
\end{itemize}
\end{itemize}
\end{theo}
\begin{proof}[Proof of Theorem \ref{thm-mainn}]
The strategy of the proof is essentially the same as that of Theorem \ref{thm-main}. 
We will only explain how to revise the proof of Theorem \ref{thm-main} to avoid repeating the same arguments. 
We use the same notations as in the proof of Theorem \ref{thm-main}. 

For an almost holomorphic map $\phi: X \dashrightarrow Y$, 
we take a resolution $\tau: \Gamma \to X$ of the indeterminacy locus of $\phi$.  
We denote, by the notation $\bar \phi: \Gamma \to Y$,  
the morphism with the following commutative diagram\,$:$
\begin{equation*}
\xymatrix@C=40pt@R=30pt{
 & \Gamma \ar[d]_\tau \ar[rd]^{\bar{\phi}\ \ }  & \\ 
& X \ar@{-->}[r]^{\phi \ \ \ }  &  Y.\\   
}
\end{equation*}
Then we have the injective sheaf morphism 
$$
\bar \phi^* K_Y \xrightarrow{\quad d \bar \phi^{*} \quad} \Lambda^m \Omega_{\Gamma}. 
$$
By taking the push-forward by the modification $\tau$, 
we  obtain the injective sheaf morphism 
$$
 (L:= \tau_* \bar \phi^* K_Y, H) \xrightarrow{\quad f \quad} 
 (\Lambda^m \Omega_{X}, \Lambda^m h). 
$$
Here we used the formula $\tau_{*}\Lambda^m \Omega_{\Gamma} =\Lambda^m \Omega_{X}$. 
For simplicity, 
we denote the line bundle $\tau_* \bar \phi^* K_Y$
by the notation $L$  and the above sheaf morphism by the notation $f$. 
We can take a (non-empty) Zariski open set $Y_1$ such that 
the restriction $\phi :X_1:=\phi^{-1}(Y_{1}) \to Y_1$ 
is a morphism without the indeterminacy locus 
since $\phi$ is an almost holomorphic map. 
Also, we take the maximal Zariski open set $Y_0 \subset Y_1$ such that 
$\phi :X_0:=\phi^{-1}(Y_{0}) \to Y_0$ is smooth. 
One of our purposes is to prove that $Y_0=Y_1$. 

From now on, we will check that the same arguments as in each step  in the proof of Theorem \ref{thm-main}  work 
by replacing $\phi^{*}K_Y$ with $L$ and $\phi :X \to Y$ with $\phi :X_1 \to Y_1$. 

By the same way as in Step \ref{step1}, 
we can construct a singular hermitian metric $H$ on $L$ and 
its dual metric $G$ on $L^{\vee}$. 
The line bundle $L=\tau_* \bar \phi^* K_Y$ 
coincides with the usual pull-back $\phi^{*}K_Y$ on $X_1$ 
(that is, $L$ can be seen as the extension of the pull-back 
$\phi^{*}K_Y$ defined on $X_1$ to $X$). 
Let $\mathcal{I}$ be the degenerate ideal of $f$ and let $Z'$ be the support of the cokernel $\mathcal{O}_X/\mathcal{I}$. 
We do not know whether $Z'$ coincides with the non-smooth locus $Z$ of $\phi$, 
but we have $Z' \cap X_1=Z \cap X_1$ 
since $f$ is just the morphism $d\phi^*$ of the pull-back on $X_1$.  
For a resolution $\pi: \bar{X} \to X$ of the degenerate ideal $\mathcal{I}$, 
we can easily check the same statements as in Claim \ref{claim-div} 
by replacing $Z$ with $Z'$. 
In particular, we have 
\begin{align*}
2\pi c_1(\pi^{*}L^\vee)
\ni 
\pi^{*}\sqrt{-1}\Theta_G=\gamma + [E], 
\end{align*}
for some smooth  $(1,1)$-form $\gamma$ and 
the integration current $[E]$ of an effective divisor $E$.

In Step \ref{step2}, we only considered tangent vectors at a smooth point of $\phi$. 
Hence there is no difficulty to obtain Claim \ref{claim-semi}.

In Step \ref{step3}, we essentially discussed local problems in $Y_1$. 
Therefore we can obtain the equality (\ref{eq-5}), the equality (\ref{eq-7}), and Claim \ref{claim-key} 
by replacing $\phi^*K_Y$ with $L$.

In Step \ref{step4}, we used the global condition that $\phi^{*}K_Y$ is pseudo-effective. 
However we can see that the line bundle $L$ is pseudo-effective by the definition, 
and thus we can repeat the same argument as in Claim \ref{claim-flat}. 
As a result, we can conclude that $L$ is numerically zero and $\sqrt{-1}\Theta_{G}$ is flat on $X$. 
This implies that $\bar \phi^* K_{Y}$ is numerically equivalent 
to some  exceptional divisor by the definition $L=\tau_* \bar \phi^* K_Y$. 
Hence the numerical dimension of $K_{Y}$ is zero. 
Further the morphism $f$ is an injective bundle morphism 
since the effective divisor $E$ is the zero divisor by $\sqrt{-1}\Theta_{G}$ is flat. 
In particular, since the morphism $f$ is equal to the morphism $d\phi^*$ of the pull-back over $Y_1$, 
the morphism $\phi$ is smooth over $Y_1$ (namely, $Y_1=Y_0$).

The rest arguments  in Step \ref{step5} are local in $Y_1=Y_0$, 
and thus we can easily check the  same conclusions as in Theorem \ref{thm-main} over $Y_{0}$  
by replacing $X$ and $Y$  with $X_0$ and $Y_0$. 
\end{proof}

\subsection{Proof of Theorem \ref{thm-mai} and Corollary \ref{cor-mainnnn}}

In this subsection, we will prove the following theorem. 
Theorem \ref{thm-mai} and Corollary \ref{cor-mainnnn} can be directly obtained from the following theorem. 

\begin{theo}\label{thm-sum}
Let $(X, g)$ be a compact K\"ahler manifold 
with semi-positive holomorphic sectional curvature and  
$Y$ be a compact K\"ahler manifold with pseudo-effective canonical bundle. 
Let $\phi: X \dashrightarrow Y$ be an almost holomorphic map from $X$ to $Y$. 

Then we have 
$$
\dim X - \dim Y \geq n_{{\rm{tf}}}{(X, g)}. 
$$
\end{theo}
 
\begin{proof}[Proof of Theorem \ref{thm-sum}]
For the proof, we will use the arguments in the proof of Theorem \ref{thm-main} and Theorem \ref{thm-mainn}. 
For simplicity, we put $k:=n_{{\rm{tf}}}{(X, g)}$. 
We take a Zariski open set $Y_1$ in $Y$ such that $\phi: X_1=\phi^{-1}(Y_1) \to Y_1$ is a morphism 
(over which $\phi$  is actually smooth by Theorem \ref{thm-mainn}). 
The  invariant $n_{{\rm{tf}}}{(X, g)}_p$ is 
lower semi-continuous with respect to $p \in X$ in the classical topology (see Definition \ref{def-flat}). 
In particular, the condition $n_{{\rm{tf}}}{(X, g)}_p=k$ is an open condition. 
Hence we can find a point $p$ such that 
$$
n_{{\rm{tf}}}{(X, g)}_p=k=\max_{p \in X}n_{{\rm{tf}}}{(X, g)}_p \text{\quad  and \quad } p \in X_1. 
$$
It follows that tangent vectors $\{e_{i}\}_{i=1}^{m}$ in the horizontal direction are truly flat by (2) in Theorem \ref{r-thm-mainn} (see also Claim \ref{claim-truly}). 
In particular, the vector space $\phi^{*}T_Y=\Spn \langle \{e_{i}\}_{i=1}^{m} \rangle$ at $p$
is contained in $V_{{\rm{flat}},p}$. 
Therefore we obtain the desired inequality $m \leq n-k$. 
\end{proof}

In the rest of this subsection, 
we will check Theorem \ref{thm-mai} and Corollary \ref{cor-mainnnn}. 

\begin{theo}[=Theorem \ref{thm-mai}]\label{r-thm-mai}
Let $(X, g)$ be a compact K\"ahler manifold such that $X$ is projective and 
the holomorphic sectional curvature  is semi-positive. 
Let  $\phi: X \dashrightarrow Y$ be a MRC fibration of $X$. 
Then we have  
$$
\dim X -\dim Y \geq n_{{\rm{tf}}}{(X, g)}. 
$$
In particular, the manifold $X$ is rationally connected 
if $n_{{\rm{tf}}}{(X, g)}=\dim X$ $($which is satisfied if the holomorphic sectional curvature is quasi-positive$)$. 
\end{theo}
\begin{proof}
We obtain the desired inequality by applying Theorem \ref{thm-sum} to MRC fibrations. 
The latter conclusion is obvious. 
Indeed, when $n_{{\rm{tf}}}{(X, g)}=\dim X$, 
the image $Y$ should be one point. 
This implies that  $X$ is rationally connected. 
\end{proof}

\begin{cor}[=Corollary \ref{cor-mainnnn}]\label{r-cor-mainnnn}
Let $(X, g)$ be a compact K\"ahler manifold 
with semi-positive holomorphic sectional curvature. 
Then we have  
$$
h^{0}(X, \Omega_X)\leq \dim X - n_{{\rm{tf}}}{(X, g)}. 
$$ 
In particular, we obtain $h^{0}(X, \Omega_X)=0$ 
if $n_{{\rm{tf}}}{(X, g)}=\dim X$ $($which is satisfied if the holomorphic sectional curvature is quasi-positive$)$. 
\end{cor}
\begin{proof}
We consider the   Albanese map  $\alpha: X \to \Alb(X)$ of $X$. 
Then the canonical bundle $K_{\Alb(X)}$ is trivial, 
and thus the assumptions in Theorem \ref{thm-sum} are satisfied. 
We obtain the desired conclusion by $\dim Y =\dim \Alb(X)=h^{0}(X, \Omega_X)$. 
\end{proof}

\subsection{Proof of Corollary \ref{cor-main}}
In this subsection, we will obtain Corollary \ref{cor-main} as an application of Theorem \ref{thm-main}. 

\begin{cor}[=Corollary \ref{cor-main}]\label{r-cor-main}
Let $X$ be a compact K\"ahler manifold with semi-positive holomorphic sectional curvature. 
Then the followings hold\,$:$
\vspace{0.2cm}\\
\quad $\bullet$ All the statements of Conjecture \ref{conj-str} hold  
in the case of $X$ being a surface. 
\vspace{0.1cm}\\
\quad $\bullet$ The  statement $(1)$ of Conjecture \ref{conj-str} holds 
if $X$ is projective and a MRC fibration of $X$ can be chosen to be a morphism. 
\end{cor}

\begin{proof}[Proof of Corollary \ref{cor-main}]

We consider a compact K\"ahler manifold $X$
with semi-positive sectional curvature. 
If the holomorphic sectional curvature is identically zero, 
then $X$ itself admits a finite \'etale cover by a complex torus by \cite{Igu54} 
(see also \cite[Proposition 2.2]{HLW16} and \cite{Ber66}). 
Then there is nothing to prove. 

From now on, we consider the case where $H_g$ is semi-positive, 
but not identically zero.
In this case, we can conclude that the canonical bundle $K_X$ is not pseudo-effective. 
Indeed, it follows that the scalar curvature $S$ of the K\"ahler metric $g$ is positive  
since the scalar curvature $S$ can be expressed 
as the integral of the holomorphic sectional curvature on 
the projective space $\mathbb{P}(T_{X,p}^\vee)$ (for example, see \cite{Ber66}). 
Then we can see that the canonical bundle $K_X$ is not pseudo-effective
by the formula 
$$
\int_{X} c_1(K_X) \wedge \omega^{n-1} = - \frac{1}{n\pi}\int_{X} S \, \omega^n < 0, 
$$
where $\omega$ is the K\"ahler form associated to $g$.

To check the first statement, we assume that $X$ is a compact K\"ahler surface. 
By the classification of compact complex surfaces, 
it can be seen that a K\"ahler surface such that $K_X$ is not pseudo-effective 
is a minimal rational surface or  a ruled surface over a curve of genus $\geq 1$. 
It is sufficient to consider the case of $X$ being a ruled surface over 
a curve of genus $\geq 1$ since a minimal rational surface is rationally connected. 
In this case, we can conclude that the ruling $X \to B$ is minimal 
(that is, a submersion) and the base is elliptic curve, 
by applying Theorem \ref{thm-main}.  
The direct summand $\phi^* T_B $ is integrable 
since the rank of $\phi^* T_B $ is one (see Remark \ref{rem-main}). 
Hence the universal cover can be shown to be 
the product of $\mathbb{C} \times \mathbb{P}^1$ 
by the classical Ehresmann theorem. 

To check the second statement, we consider a smooth projective variety 
whose holomorphic sectional curvature is semi-positive, but not identically zero. 
Then, since $K_Y$ not pseudo-effective by the first half argument, 
a MRC fibration $\phi: X \dashrightarrow Y$ is non-trivial. 
Then the image $Y$ is not uniruled by \cite[Theorem 1.1]{GHS03} and 
the canonical bundle $K_Y$ of $Y$ is pseudo-effective by \cite{BDPP}. 
Therefore we can directly apply Theorem \ref{thm-main} if a MRC fibration can be chosen to be a morphism. 
Then the statement $(1)$ of Conjecture \ref{conj-str} is obvious. 
\end{proof}

In the rest of subsection, we give a remark on smooth projective varieties with nef anti-canonical bundle. 
Even if a compact K\"ahler manifold $X$ has the semi-positive holomorphic sectional curvature, 
the anti-canonical bundle $K_X^{\vee}$ is not necessarily nef (for example, see \cite[Example 3.6]{Yan16}). 
However it is worth to mention that we can confirm that Conjecture \ref{conj-str} holds 
when $X$ is projective and $X$ has the nef anti-canonical bundle, 
by using Theorem \ref{thm-main} and the deep structure theorem proved by Cao-H\"oring in \cite{CH17}.
\begin{cor}\label{cor-CH}
Let $(X, g)$ be a compact K\"ahler manifold such that $X$ is projective and the holomorphic sectional curvature is semi-positive. 
Further we assume that the anti-canonical bundle $K_X^\vee$ is nef. 
Then Conjecture \ref{conj-str} can be affirmatively solved. 
\end{cor}
\begin{proof}
By the structure theorem of \cite{CH17}, we can choose a MRC fibration $\phi: X \to Y$ to be a (locally trivial) morphism. 
Further we have the decomposition $X_{\rm{univ}}\cong F \times Y_{\rm{univ}}$, 
where $F$ is the rationally connected fiber of $\phi$. 
By applying Theorem \ref{thm-main} to this MRC fibration $\phi: X \to Y$, 
we can see that $Y$ admits a finite \'etale  cover $T \to Y$ by an abelian variety $T$. 
This finishes the proof. 
\end{proof}

\section{Open problems related to semi-positive sectional curvature}\label{Sec-4}
In this section, we discuss open problems related to the geometry of semi-positive sectional curvature. 

The first problem is concerned with Conjecture \ref{conj-str}. 
If (1) and (2) in the following problem are affirmatively solved, 
then Conjecture \ref{conj-str} for smooth projective varieties 
can be obtained from Theorem \ref{thm-main}. 

\begin{prob}\label{prob-main}
Let $(X, g)$ be a compact K\"ahler manifold such that $X$ is projective and 
the holomorphic sectional curvature $H_g$ is  semi-positive. 
Let $\phi : X \dashrightarrow Y$ be a MRC fibration of $X$. 
\begin{itemize}
\item[(1)] Can we choose a MRC fibration of $X$ to be a morphism?
\item[(2)] Does $\phi^{*}T_Y$  determine an integrable foliation? 
\item[(3)] Does the fiber $F$ admit a K\"ahler metric $g_F$ such that $n_{{\rm{tf}}}(F, g_F)=\dim F$?
\item[(4)] Does the equality $\dim X = \dim Y + n_{{\rm{tf}}}(X, g)$ hold?
\end{itemize}
\end{prob}

The following problem seems to give a strategy to study Conjecture \ref{conj-str} for compact K\"ahler manifolds. 
If the following problem can be solved, 
we can   apply Theorem \ref{thm-main} and Conjecture \ref{conj-str} to MRC fibrations of  the fiber $Z$.

\begin{prob}\label{prob-mainn}
Let $(X, g)$ be a compact K\"ahler manifold 
with the semi-positive holomorphic sectional curvature $H_g$. 
After we take a suitable finite \'etale cover $X^* \to X$, 
we consider the  Albanese map $\alpha: X^* \to \Alb(X^*)$. 
\begin{itemize}
\item[(1)] Is the fiber $Z$ projective? 
\item[(2)] Is the holomorphic sectional curvature 
$H_{g_Z}$ of the induced metric $g_Z$ semi-positive? 
\end{itemize}
\end{prob}

When $X$ admits a K\"ahler metric with quasi-positive holomorphic sectional curvature, 
it seems to be natural to expect that $X$ is automatically projective 
(cf. \cite[Theorem 1.7]{Yan17a}). 
The following problem, which was posed by Yang in a private discussion, 
gives a generalization of this expectation. 
Also, it is an interesting problem to consider 
rationally connectedness or holomorphic sectional curvature 
from the viewpoint of (uniform) RC positivity introduced by Yang. 
See \cite{Yan18b} for vanishing theorems and 
\cite[Theorem 1.7, Conjecture 1.9]{Yan18c} for rationally connectedness.

\begin{prob}\label{prob-proj}
Let $(X, g)$ be a compact K\"ahler manifold $($or more generally a hermitian manifold$)$
with the semi-positive holomorphic sectional curvature $H_g$. 
Assume that $X$ has no  truly flat vector at some point of $X$ 
$($or  $H_g$ is quasi-positive$)$. 
\begin{itemize}
\item[(1)] Can we obtain $h^{i}(X, \mathcal{O}_X)=0$ for any $i>0$? 
\item[(2)] Is $X$ automatically projective and rationally connected? 
\end{itemize}
\end{prob}


\end{document}